\documentclass[12pt]{amsart}
\usepackage{amssymb}
\usepackage{url}

\newcommand{\Z}{{\mathbb{Z}}}
\newcommand{\R}{{\mathbb{R}}}
\newcommand{\C}{{\mathbb{C}}}

\newcommand{\slm}{{\mathfrak{sl}}}
\newcommand{\fg}{{\mathfrak{g}}}
\newcommand{\fh}{{\mathfrak{h}}}
\newcommand{\be}{{\mathbf{e}}}
\newcommand{\cL}{{\mathcal{L}}}
\newcommand{\vQ}{{\vec{\mathcal{I}}}}
\newcommand{\vQp}{{\vec{\mathcal{I}'}}}

\newcommand{\hgt}{{\operatorname{ht}}}
\renewcommand{\leq}{\leqslant}
\renewcommand{\geq}{\geqslant}

\newtheorem{thm}{Theorem}[section]
\newtheorem{lem}[thm]{Lemma}
\newtheorem{prop}[thm]{Proposition}
\newtheorem{cor}[thm]{Corollary}

\theoremstyle{definition}
\newtheorem{defn}[thm]{Definition}
\newtheorem{exmp}[thm]{Example}
\newtheorem{abs}[thm]{}

\theoremstyle{remark}
\newtheorem{rem}[thm]{Remark}

\numberwithin{equation}{section}

\begin{document}
\title{On Lusztig's canonical bases of simple Lie algebras}
\author{Meinolf Geck}
\address{Lehrstuhl f\"ur Algebra\\Universit\"at Stuttgart\\
Pfaffenwaldring 57\\D--70569 Stuttgart\\ Germany}
\curraddr{}
\email{meinolf.geck@mathematik.uni-stuttgart.de}
\thanks{This work is a contribution to the SFB-TRR 195 ``Symbolic Tools in
Mathematics and their Application'' of the German Research Foundation 
(DFG)}


\subjclass[2000]{Primary 20G40; Secondary 17B45}

\date{}

\begin{abstract} Let $\fg$ be a simple Lie algebra over~$\C$ with root 
system~$\Phi$. In the simply laced case, Frenkel and Kac found a 
particularly simple construction of~$\fg$, together with a Chevalley 
basis and explicitly given structure constants, in terms of a certain 
multiplicative $2$-cocycle $\varepsilon\colon \Z \Phi\times \Z\Phi 
\rightarrow\{\pm 1\}$. We show that Lusztig's canonical basis of~$\fg$ 
can also be obtained in this way, for a suitable choice of~$\varepsilon$. 
We also address the problem of explicitly describing the structure constants 
when $\Phi$ is not simply laced.
\end{abstract}

\keywords{Root systems, Lie algebras, canonical bases, structure constants}
\maketitle

\section{Introduction} \label{sec0}

Let $\fg$ be a finite-dimensional simple Lie algebra over $\C$. In this
paper we address some questions concerning the various possible choices 
of Chevalley bases for~$\fg$, and the determination of the corresponding 
structure constants. Beginning with Chevalley \cite{Ch}, this subject has 
a long history; here it suffices to cite the work of Tits \cite{Ti}, 
Frenkel--Kac \cite{FK}, Ringel \cite{Ri} and, finally, Lusztig \cite{L5}. 
(Further references will be given in due course in the main text.)

To explain what we are trying to do, let us first assume that the root
system~$\Phi$ of~$\fg$ is simply laced, of type $A_r$ ($r\geq 1$), $D_r$ 
($r\geq 3$) or $E_r$ ($r=6,7,8$). Then Frenkel and Kac \cite{FK} found a 
particularly simple way~---~compared to the general treatment of Tits
\cite{Ti}~---~of constructing~$\fg$ together with a Chevalley basis. This 
leads to an explicit description of the corresponding structure constants 
in terms of a certain multiplicative cocycle $\varepsilon\colon\Z\Phi\times
\Z\Phi \rightarrow \{\pm 1\}$, which can be characterised by a few formal 
conditions motivated by the consideration of central extensions of the
lattice $\Z \Phi$ by $\Z/2\Z$. There are various choices in this 
construction; for example, each orientation of the Dynkin diagram of 
$\Phi$ gives rise to a different~$\varepsilon$. On the other hand, 
there is Lusztig's general theory \cite{L10} of canonical bases for
quantised enveloping algebras. Viewing $\fg$ as a $\fg$-module via the 
adjoint representation, that theory gives rise to a canonical basis of 
$\fg$ itself. An elementary description of that basis, without reference 
to the theory of quantised enveloping algebras, can be found in the 
author's lecture notes \cite[\S 2.7]{mynotes}. 

The first aim of this paper is to establish a relation between the 
Frenkel--Kac construction and Lusztig's canonical basis. Once this is 
achieved, we obtain~---~as an immediate by-product~---~explicit formulae 
for the structure constants with respect to Lusztig's basis of~$\fg$. In 
fact, such formulae were previously found by a different method in the 
Master's thesis of Lang (see \cite{GeLa}, \cite{Lang}), without reference 
to the Frenkel--Kac construction. So the approach in this paper 
provides~---~at the same time~---~a new proof for these formulae, and also 
a new proof for the existence of Lusztig's canonical basis for~$\fg$. 

The second aim of this paper is to consider the general case where the
root system of $\fg$ may not be simply laced. Then $\fg$ can be obtained
by a ``folding'' procedure from suitable simply laced cases. Inspired by 
the work of Ringel \cite{Ri}, and unpublished notes by Rylands \cite{Ry}, 
we will determine rather simple formulae for the structure constants with 
respect to Lusztig's canonical basis in this case as well. In a certain
sense (made more precise in Remark~\ref{summa}), these formulae only 
depend on the type of $\fg$, but not on the rank or a specific labelling
of the Dynkin diagram. The exact formulations were found by extensive 
experimentation with the author's programs in \cite{mygreen}, 
\cite{chevlie}, but the final proofs will be computer-free.

This paper is organised as follows. In Section~\ref{seccan} we formulate an 
axiomatic setting for dealing with Lusztig's canonical basis; this may be of 
independent interest and it will also be useful in the subsequent discussions. 
In Section~\ref{secflm} we begin by briefly explaining the Frenkel--Kac 
construction, where we follow the slightly more general exposition of
Frenkel--Lepowsky--Meurman \cite[Chap.~6]{FLM}. The cocycles~$\varepsilon$ 
that are most common in the literature, e.g., in De Graaf 
\cite[\S 5.13]{graaf}, Kac \cite[\S 7.8]{kac} or Springer \cite[\S 10.2]{Sp},
are typically normalised such that 
\begin{center}
$\varepsilon(\alpha,\alpha)=-1 \qquad$ for all $\alpha\in \Phi$.
\end{center}
One easily sees that no such $\varepsilon$ will give 
rise to Lusztig's canonical basis. However, based on an idea from Lang's 
Master's thesis (see \cite{GeLa}, \cite{Lang}), there are natural choices 
of cocycles $\varepsilon_0 \colon\Z\Phi\times\Z\Phi\rightarrow\{\pm 1\}$
which both satisfy the requirements for the Frenkel--Kac construction, 
and which are normalised such that 
\begin{center}
$\varepsilon_0(\alpha,\alpha)=-(-1)^{\text{height}(\alpha)} \qquad$
for all $\alpha\in \Phi$.
\end{center}
In Proposition~\ref{propGL} we shall see that cocycles $\varepsilon_0$ as 
above indeed give rise to Lusztig's canonical basis; in fact, there are 
precisely two canonical choices for~$\varepsilon_0$, which correspond to 
the two possible orientations of the Dynkin diagram in which each vertex 
is either a sink or a source. Finally, Section~\ref{secnon} deals with Lie 
algebras with a non-simpy laced root system, where the combinatorial details
needed for $\fg$ of type $B_r$ and~$C_r$ are contained in a 
separate Section~\ref{secBC}.

\section{Special Chevalley systems} \label{seccan}

Let $E$ be a finite-dimensional $\R$-vector space and $\langle\;,\;\rangle
\colon E \times E \rightarrow \R$ be a positive-definite scalar product.
Let $\Phi\subseteq E$ be an (irreducible) crystallographic root system 
such that $E$ is spanned by $\Phi$. We fix a system $\{\alpha_i\mid i 
\in I\}$ of simple roots, where $I$ is an index set such that $\dim E=|I|$. 
Thus, every $\alpha\in \Phi$ can be written uniquely as a $\Z$-linear 
combination of $\{\alpha_i \mid i \in I\}$, where either all coefficients 
are $\geq 0$, or all coefficients are $\leq 0$. Accordingly, we have a 
partition $\Phi=\Phi^+ \sqcup \Phi^-$ into positive and negative roots. 
Let $A=(a_{ij})_{i,j\in I}$ be the corresponding Cartan matrix; we have 
\[ a_{ij}=2\frac{\langle \alpha_i,\alpha_j\rangle}{\langle \alpha_i,\alpha_i
\rangle} \in \Z\qquad \mbox{for all $i,j\in \Z$}.\]
See Humphreys \cite[Part~III]{H} for further background. For any $\alpha\in 
\Phi$ let us set $\alpha^\vee:=2\alpha/\langle \alpha, \alpha\rangle\in E$. 
Then $\langle \alpha^\vee,\beta\rangle=q_{\alpha,\beta}-p_{\alpha,\beta}$
for $\alpha,\beta\in \Phi$ such that $\alpha\neq \pm \beta$, where 
\begin{align*}
p_{\alpha,\beta}&:=\max\{m \geq 0 \mid \beta+m\alpha\in \Phi\},\\
q_{\alpha,\beta}&:=\max\{m \geq 0 \mid \beta-m\alpha\in \Phi\};
\end{align*}
see, e.g., Humphreys \cite[\S 8.5]{H} or \cite[Lemma~2.6.2]{mynotes}. The 
integers $p_{\alpha,\beta}$ and $q_{\alpha,\beta}$ play a crucial role in 
the description of Chevalley bases.

Let $\fg$ be a simple Lie algebra over $\C$ with root system $\Phi$. We 
recall some facts from the general structure theory of $\fg$; see again 
\cite{H} for further background. Let $\fh\subseteq \fg$ be a Cartan 
subalgebra. We have a corresponding decomposition $\fg=\fh\oplus 
\bigoplus_{\alpha\in \Phi}\fg_\alpha$ where each $\fg_\alpha$ is a 
$1$-dimensional subspace. For each $\alpha\in \Phi$, there is a 
well-defined co-root $h_\alpha\in \fh$ which is uniquely characterised 
by the conditions that $0\neq h_\alpha\in [\fg_\alpha,\fg_{-\alpha}]$ and 
that $[h_\alpha,x]=2x$ for all $x\in \fg_\alpha$. For each $\alpha\in\Phi$ 
let us choose a non-zero ``root element'' $0\neq e_\alpha \in \fg$. Then 
\[ [e_\alpha,e_{-\alpha}]=\zeta_\alpha h_\alpha \qquad \mbox{where
$0\neq \zeta_\alpha\in \C$}.\]
(For the time being, we do not impose any conditions on $\zeta_\alpha$;
just note that $\zeta_{\alpha}=\zeta_{-\alpha}$ for all $\alpha\in \Phi$,
since $h_{-\alpha}=-h_\alpha$.)

For $i\in I$ we write $h_i:=h_{\alpha_i}$. Then $\cL:=\{e_\alpha \mid
\alpha\in \Phi\}$ together with $\{h_i\mid i \in I\}$ forms a basis 
of~$\fg$. For $\alpha,\beta\in \Phi$ such that $\alpha+\beta\in 
\Phi$ we write as usual
\[ [e_\alpha,e_\beta]=N_{\alpha,\beta}e_{\alpha+\beta} \qquad
\mbox{where} \qquad N_{\alpha,\beta} \in \C.\] 
Chevalley \cite{Ch} proved the fundamental result that the collection
$\cL=\{e_\alpha \mid\alpha\in \Phi\}$ can be chosen such that 
\begin{equation*}
N_{\alpha,\beta}=\varepsilon(\alpha,\beta)(q_{\alpha,\beta}+1) \qquad 
\mbox{if $\alpha,\beta,\alpha+\beta\in \Phi$}, \tag{$\clubsuit$}
\end{equation*}
where $\varepsilon(\alpha,\beta)\in \{\pm 1\}$. In this case, we say
that $\cL$ is a \textit{Chevalley system} for~$\fg$ and call 
$\varepsilon$ the ``sign function associated with~$\cL$''. 

Note that, since we are not imposing any conditions on the 
scalars~$\zeta_\alpha$, our definition of a ``Chevalley system'' is 
somewhat more general than that in Bourbaki \cite[Ch.~VIII, \S 2, 
no.~4]{bour2}. 

Clearly, ($\clubsuit$) alone does not determine $\cL$ uniquely; for 
example, we may replace each $e_\alpha$ by $\pm e_\alpha$ (where the
sign may depend on~$\alpha$) and the above relations will still hold, 
with a possibly different sign function $\varepsilon(\alpha,\beta)$. But
of course, the anti-symmetry of $[\;,\;]$ and the Jacobi identity impose 
certain conditions. Here is a very useful example for such a condition.

\begin{lem} \label{cass1} Let $\cL$ be a Chevalley system, as above. Let 
$\gamma_1,\gamma_2,\gamma_2\in \Phi$ be such that $\gamma_1+\gamma_2+
\gamma_3=0$. Then $\varepsilon(\gamma_1,\gamma_2)=-\varepsilon(\gamma_2,
\gamma_1)$ and 
\[ \zeta_{\gamma_3}\varepsilon(\gamma_1,\gamma_2)=\zeta_{\gamma_1}
\varepsilon(\gamma_2,\gamma_3)=\zeta_{\gamma_2}\varepsilon(\gamma_3,
\gamma_1).\]
\end{lem}

\begin{proof} The first formula is clear by the anti-symmetry of
$[\;,\;]$. Using the Jacobi identity and arguing as in Carter 
\cite[p.~52]{C1}, we find that 
\[ \frac{\zeta_{\gamma_3}}{\langle \gamma_3,\gamma_3\rangle} N_{\gamma_1,
\gamma_2}=\frac{\zeta_{\gamma_1}}{\langle \gamma_1,\gamma_1\rangle}
N_{\gamma_2,\gamma_3}=\frac{\zeta_{\gamma_2}}{\langle \gamma_2,\gamma_2
\rangle} N_{\gamma_3,\gamma_1}.\]
Now, by Casselman \cite[Lemma~2.3]{Cass0}, we also have 
\[ \frac{q_{\gamma_1,\gamma_2}+1}{\langle \gamma_3,\gamma_3\rangle}
=\frac{q_{\gamma_2,\gamma_3}+1}{\langle \gamma_1,\gamma_1\rangle}=
\frac{q_{\gamma_3,\gamma_1}+1}{\langle \gamma_2,\gamma_2 \rangle}.\]
Then the second desired formula follows using ($\clubsuit$).
\end{proof}

\begin{rem} \label{extras} One way of fixing a Chevalley system is 
described by Carter \cite[p.~58--59]{C1}. This relies on the choice of a 
positive system $E^+\subseteq E$ as in \cite[\S 2.1]{C1}. Correspondingly, 
one can define the notions of ``special'' and ``extraspecial'' pairs of 
roots. Then there is a unique Chevalley system $\cL=\{e_\alpha\mid\alpha 
\in \Phi\}$ such that $[e_\alpha,e_{-\alpha}]=h_\alpha$ for all $\alpha\in
\Phi$, and $\varepsilon(\alpha,\beta)=1$ whenever the pair $(\alpha,
\beta)$ is extraspecial. Such Chevalley systems have been used quite 
frequently in the literature, especially in a computational context; see, 
e.g., Cohen--Murray--Taylor \cite{Coet}, Gilkey--Seitz \cite{GiSe} or 
Vavilov \cite{vav2}. 
\end{rem}

Following Lusztig \cite{L1,L2,L5} we will now single out a set of 
conditions on $\cL$ which severely restrict the possibilities for the 
corresponding signs functions $\varepsilon(\alpha,\beta)$ and which, as far
as possible, do not depend on any choices. The conditions in the following
definition look slightly different from those in \cite{L1,L2,L5}; the exact 
relation between the two will be explained in Theorem~\ref{thmcan} below.

\begin{defn} \label{defcan} Let $\cL=\{e_\alpha\mid \alpha\in \Phi\}$ be 
a collection of root elements such that~($\clubsuit$) holds. We say that
$\cL$ is a ``special Chevalley system'' (or a ``Lusztig system'') if, 
for each $i\in I$, we have $[e_{\alpha_i},e_{-\alpha_i}]=h_i$ and 
there is a sign $c_i\in \{\pm 1\}$ such that 
\begin{alignat*}{2}
[e_{\alpha_i},e_\alpha]&=c_i(q_{\alpha_i,\alpha}+1)e_{\alpha+\alpha_i} &&
\quad \mbox{if $\alpha\in \Phi^+$ and $\alpha+\alpha_i\in \Phi$},\tag{a}\\
[e_{-\alpha_i},e_\alpha]&=c_i(p_{\alpha_i,\alpha}+1)e_{\alpha-\alpha_i} 
&&\quad \mbox{if $\alpha\in \Phi^-$ and $\alpha-\alpha_i\in \Phi$}.\tag{b}
\end{alignat*}
(Note that $p_{\alpha_i,\alpha}=q_{-\alpha_i,\alpha}$ so (b) is
consistent with $(\clubsuit)$.) 
\end{defn}

\begin{rem} \label{remcan1} Let $\cL$ be a special Chevalley system as
above, with corresponding signs $\{c_i\mid i \in I\}$. 

(a) Let $e_\alpha^-:=-e_\alpha$ for all $\alpha\in \Phi$. Then 
$\cL^-:=\{e_\alpha^-\mid \alpha\in \Phi\}$ also is a special Chevalley 
system, with corresponding signs $\{-c_i\mid i \in I\}$. If $\alpha,
\beta\in \Phi$ are such that $\alpha+\beta \in \Phi$, then we have 
\[ [e_\alpha^-,e_\beta^-]=-\varepsilon(\alpha,\beta)(q_{\alpha,\beta}+1)
e_{\alpha+\beta}^-,\]
where $\varepsilon(\alpha,\beta)$ is the sign function of $\cL$.

(b) The signs $\{c_i\}$ satisfy the following condition.
\[ i,j\in I, \quad i \neq j \quad\mbox{and} \quad a_{ij} \neq 0
\quad \Rightarrow \quad c_i=-c_j.\]
Indeed, let $i \neq j$ such that $a_{ij}\neq 0$. Since $\alpha_i-
\alpha_j\not\in \Phi$, we have $q_{\alpha_i,\alpha_j}=q_{\alpha_j,
\alpha_i}=0$. Since $a_{ij}=q_{\alpha_i,\alpha_j}-p_{\alpha_i,\alpha_j}$, 
we conclude that $p_{\alpha_i,\alpha_j}>0$ and so $\alpha_i+\alpha_j 
\in\Phi$. Using twice the identity in Definition~\ref{defcan}(a), we obtain 
\begin{align*}
[e_{\alpha_i},e_{\alpha_j}]&=c_i(q_{\alpha_i,\alpha_j}+1)
e_{\alpha_i+\alpha_j}=c_ie_{\alpha_i+\alpha_j} ,\\
[e_{\alpha_j},e_{\alpha_i}]&=c_j(q_{\alpha_j,\alpha_i}+1)
e_{\alpha_j+\alpha_i}=c_je_{\alpha_i+\alpha_j}.
\end{align*}
Since $[\;,\;]$ is anti-symmetric, this forces $c_i=-c_j$, as claimed.
(Collections of signs satisfying the above condition already appeared 
in the work of Rietsch \cite[4.1]{KR}.) 
\end{rem}

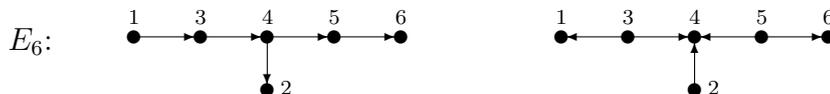
\begin{table}[htbp] \caption{Two orientations of the graph of type $E_6$} 
\label{Morient} \begin{center} \makeatletter
\begin{picture}(350,45)
\put( 13, 35){$E_6$:}
\put( 60, 40){\circle*{5}}
\put( 58, 45){$\scriptstyle{1}$}
\put( 60, 40){\vector(1,0){23}}
\put( 85, 40){\circle*{5}}
\put( 83, 45){$\scriptstyle{3}$}
\put( 85, 40){\vector(1,0){23}}
\put(110, 40){\circle*{5}}
\put(108, 45){$\scriptstyle{4}$}
\put(110, 40){\vector(0,-1){18}}
\put(110, 20){\circle*{5}}
\put(115, 18){$\scriptstyle{2}$}
\put(110, 40){\vector(1,0){23}}
\put(135, 40){\circle*{5}}
\put(133, 45){$\scriptstyle{5}$}
\put(135, 40){\vector(1,0){23}}
\put(160, 40){\circle*{5}}
\put(158, 45){$\scriptstyle{6}$}

\put(220, 40){\circle*{5}}
\put(218, 45){$\scriptstyle{1}$}
\put(245, 40){\vector(-1,0){23}}
\put(245, 40){\circle*{5}}
\put(243, 45){$\scriptstyle{3}$}
\put(245, 40){\vector(1,0){23}}
\put(270, 40){\circle*{5}}
\put(268, 45){$\scriptstyle{4}$}
\put(270, 20){\vector(0,1){18}}
\put(270, 20){\circle*{5}}
\put(275, 18){$\scriptstyle{2}$}
\put(295, 40){\vector(-1,0){23}}
\put(295, 40){\circle*{5}}
\put(293, 45){$\scriptstyle{5}$}
\put(295, 40){\vector(1,0){23}}
\put(320, 40){\circle*{5}}
\put(318, 45){$\scriptstyle{6}$}
\put( 03,  0){\footnotesize (The numbers attached to the vertices define 
a standard labelling of the graph.)}
\end{picture}
\end{center}
\end{table}

\begin{rem} \label{graphA} Consider the ``incidence graph'' of $A$. It has
vertices labelled by the index set~$I$; for $i \neq j$ in $I$, there is a
simple edge between $i$ and $j$ if $a_{ij}\neq 0$. (Note that $a_{ij}\neq 0
\Leftrightarrow a_{ji}\neq 0$.) Putting arrows on the edges of this graph,
we obtain an ``orientation'' of the graph; see Table~\ref{Morient} for 
two examples. Now a collection of signs $\{c_i\}$ satisfing the condition 
in Remark~\ref{remcan1}(b) precisely characterises an orientation where 
each vertex is either a sink or a source. (For example, the second graph
in Table~\ref{Morient} has this property.) Given $i \neq j$ in $I$ that 
are joined by an edge (i.e., $a_{ij}\neq 0$), we have an arrow on that 
edge pointing from~$i$ to~$j$ if and only if $c_i=-1$ and $c_j=1$. (For 
example, the orientation of the second graph in Table~\ref{Morient} is 
characterised by the signs $c_2=c_3=c_5=-1$ and $c_1=c_4=c_6=1$.) Since 
the graph of $A$ is a connected tree, it is clear that there only two 
orientations where each vertex is either a sink or a source. Hence, there 
are only two collections of signs $\{c_i\}$ as above, obtained one from 
another by replacing each $c_i$ by $-c_i$.
\end{rem}

\begin{prop} \label{propcan1} Let $\cL$ be a special Chevalley system as
above, with corresponding signs $c_i$ ($i\in I$) and $\varepsilon(\alpha,
\beta)$ (for $\alpha,\beta\in \Phi$ such that $\alpha+\beta\in \Phi$). 
Then the following hold.
\begin{itemize}
\item[(a)] The signs $\{\varepsilon(\alpha,\beta)\}$ only depend on 
the signs $\{c_i\}$.
\item[(b)] We have $\varepsilon(\alpha,\beta)=\varepsilon(-\alpha,-\beta)$
for $\alpha,\beta,\alpha+\beta\in \Phi$. 
\end{itemize}
\end{prop}

\begin{proof} In the following, it will be convenient to set $e_i:=c_i
e_{\alpha_i}$ and $f_i:=c_ie_{-\alpha_i}$ for $i \in I$; note that we 
still have $[e_i,f_i]=h_i$. Let $\alpha \in \Phi^+$. Then it is well-known 
that there is a sequence of indices $i_1,\ldots,i_r\in I$ ($r\geq 1$) such 
that $\beta_j:=\alpha_{i_1}+\ldots +\alpha_{i_j}\in \Phi$ for $1\leq j
\leq r$, where $\beta_r=\alpha$; see, e.g., \cite[\S 10.2]{H}. Let us fix 
such a sequence for each $\alpha \in \Phi^+$. Then 
Definition~\ref{defcan}(a) shows that 
\[e_\alpha=(m_\alpha^+)^{-1}[e_{i_r},\ldots [e_{i_3},[e_{i_2},e_{i_1}]]
\ldots ]\]
where $0\neq m_\alpha^+\in \Z$ equals $c_{i_1}$ times the product of all 
$c_{i_j} (q_{\alpha_{i_j},\beta_{j-1}}+1)$ for $2\leq j \leq r$.  
Similarly, Definition~\ref{defcan}(b) shows that 
\[e_{-\alpha}=(m_\alpha^-)^{-1}[f_{i_r},\ldots [f_{i_3},[f_{i_2},f_{i_1}]]
\ldots ]\]
where $0\neq m_\alpha^-\in \Z$ equals $c_{i_1}$ times the product of all 
$c_{i_j}(p_{\alpha_{i_j},-\beta_{j-1}}+1)$ for $2\leq j \leq r$. 
Note that $p_{\alpha_{i_j},-\beta_{j-1}}=q_{\alpha_{i_j},\beta_{j-1}}$
for $2\leq j \leq r$, so we actually have $m_\alpha^-=m_\alpha^+$.

(a) Let $\cL'=\{e_\alpha' \mid \alpha\in \Phi\}$ also be a special 
Chevalley system, with respect to the same signs $\{c_i\mid i\in I\}$. 
Write 
\[[e_\alpha',e_\beta']=\varepsilon'(\alpha,\beta)(q_{\alpha,\beta}+1) 
e_{\alpha+\beta}' \qquad \mbox{if $\alpha,\beta,\alpha+\beta\in \Phi$},\]
where $\varepsilon'(\alpha,\beta)\in \{\pm 1\}$. We must show that 
$\varepsilon(\alpha,\beta)=\varepsilon'(\alpha,\beta)$. To see this,
we set again $e_i':= c_i e_{\alpha_i}'$ and $f_i':=c_ie_{-\alpha_i}'$ 
for $i \in I$. Now we use the Isomorphism Theorem for semisimple Lie 
algebras; see, e.g., \cite[Theorem~2.7.14]{mynotes} which contains a proof 
well suited to the setting here. This shows that there exists some $\theta
\in\mbox{Aut}(\fg)$ such that $\theta(e_i)=e_i'$ and $\theta(f_i)=f_i'$ for
all $i\in I$. Now let $\alpha\in \Phi$ and apply $\theta$ to $e_\alpha$, 
expressed as above in terms of iterated Lie brackets. Since there are 
analogous such expressions for~$e_{\alpha}'$, with the same coefficients 
$m_\alpha^{\pm}$~(!), it follows that $\theta (e_\alpha)=e_\alpha'$ for 
all $\alpha\in \Phi$. Hence, the structure constants of $\fg$ with respect 
to $\cL$ and to $\cL'$, respectively, are the same. In particular, we 
have $\varepsilon'(\alpha,\beta)=\varepsilon(\alpha,\beta)$ if $\alpha,
\beta,\alpha+\beta\in \Phi$, as desired. 

(b) Note that the desired identity certainly holds if $\alpha=\alpha_i$
for some $i\in I$, by the two formulae in Definition~\ref{defcan}. 
In general, we use again the Isomorphism Theorem (already cited above).
This shows that there exists some $\omega \in \mbox{Aut}(\fg)$ such that 
$\omega(e_i)=f_i$ and $\omega(f_i)=e_i$ for all $i\in I$. Applying 
$\omega$ to $e_\alpha$ expressed as above, we see that $\omega(e_\alpha)=
e_{-\alpha}$ for all $\alpha\in \Phi$. This immediately yields the 
desired statement. 
\end{proof}

\begin{cor} \label{cancor} For a given simple Lie algebra $\fg$ with 
root system $\Phi$, there are only two possible sign functions 
$\varepsilon(\alpha,\beta)$ (for $\alpha,\beta\in \Phi$ such that 
$\alpha+\beta\in \Phi$) arising from a special Chevalley system $\cL$ 
for $\fg$; if $\varepsilon$ is one of them, then the other one 
is~$-\varepsilon$. 
\end{cor}

\begin{proof} Let $\cL=\{e_\alpha\mid \alpha\in \Phi\}$ be a special 
Chevalley system, with corresponding signs $\{c_i\}$ and $\{\varepsilon
(\alpha,\beta)\}$. Let also $\cL'=\{e_\alpha'\mid \alpha \in \Phi\}$ be 
a special Chevalley system, with signs $\{c_i'\}$ and $\{\varepsilon'
(\alpha,\beta)\}$. By Remark~\ref{graphA}, there exists some $\delta\in
\{\pm 1\}$ such that $c_i'=\delta c_i$ for all $i\in I$. If $\delta=-1$, 
then replace $\cL$ by $\cL^-$ as in Remark~\ref{remcan1}(a). Thus, we may 
assume without loss of generality that $c_i'=c_i$ for all $i\in I$. But 
then the corresponding sign functions $\varepsilon(\alpha,\beta)$ and 
$\varepsilon'(\alpha,\beta)$ are the same by Proposition~\ref{propcan1}(a).
\end{proof}

\begin{lem} \label{lem21} Let $\cL=\{e_\alpha \mid \alpha\in \Phi\}$ be 
a special Chevalley system. Let $i\in I$ and $\alpha\in \Phi$ be such 
that $\alpha-\alpha_i\in \Phi$. Then 
\[ \zeta_\alpha=-\zeta_{\alpha-\alpha_i} \qquad \mbox{and}\qquad 
\varepsilon(-\alpha_i,\alpha)=\varepsilon(\alpha_i,\alpha-\alpha_i).\]
\end{lem}

\begin{proof} We set $\beta:=\alpha-\alpha_i \in \Phi$ and $\xi:=
\varepsilon(-\alpha_i,\alpha)(q_{-\alpha_i,\alpha}+1)$. Then 
$[e_{-\alpha_i},e_{\alpha}]=\xi e_{\alpha-\alpha_i}=\xi e_{\beta}$ by 
($\clubsuit$). On the other hand, since $\beta+\alpha_i= \alpha\in \Phi$, 
we also have 
\[ [e_{\alpha_i},[e_{-\alpha_i},e_\alpha]]=\xi [e_{\alpha_i},
e_\beta]=\xi \varepsilon(\alpha_i,\beta)(q_{\alpha_i,\beta}+1)
e_{\alpha}.\]
Now, since $e_{\pm \alpha_i}\in \fg_{\pm \alpha_i}$ and $[e_{\alpha_i},
e_{-\alpha_i}]=h_i$, the subspace $\fg_i \subseteq \fg$ spanned by 
$e_{\alpha_i},e_{-\alpha_i},h_i$ is isomorphic to the Lie algebra
$\slm_2(\C)$. Regarding $\fg$ as a $\fg_i$-module and using standard 
facts from the representation theory of $\slm_2(\C)$, it follows that 
\[ [e_{\alpha_i},[e_{-\alpha_i},e_\alpha]]=q_{\alpha_i,\alpha}(p_{\alpha_i,
\alpha}+1) e_\alpha;\]
see, e.g., \cite[Remark~2.2.15(c)]{mynotes} or \cite[\S 7.2]{H}.
Hence, we conclude that
\begin{align*}
q_{\alpha_i,\alpha}&(p_{\alpha_i,\alpha}+1)=\xi \varepsilon(\alpha_i, 
\beta)(q_{\alpha_i,\beta}+1)\\ &=\varepsilon(-\alpha_i,\alpha)\varepsilon
(\alpha_i, \beta)(q_{-\alpha_i,\alpha}+1)(q_{\alpha_i,\beta}+1).
\end{align*}
Since $p_{\alpha_i,\alpha}$, $q_{\alpha_i,\alpha}$, $q_{\alpha_i,\beta}$ 
and $q_{-\alpha_i,\alpha}$ are in $\Z_{\geq 0}$, the right hand side of 
the above identity is non-zero and so we must have $\varepsilon(-\alpha_i,
\alpha)=\varepsilon(\alpha_i,\beta)$, as claimed. 

Now consider the identity $\zeta_\alpha=-\zeta_\beta$. Since 
$\zeta_{-\alpha}= \zeta_\alpha$ and $\zeta_{-\beta}=\zeta_\beta$, it is 
enough to prove this for the case where $\alpha\in \Phi^+$; note that then 
also $\beta\in \Phi^+$. We use Lemma~\ref{cass1} with $\gamma_1:=-\alpha_i$, 
$\gamma_2:=-\beta$ and $\gamma_3:=\alpha$. This yields the identity:
\[ \zeta_\alpha\varepsilon(-\alpha_i,-\beta)=\zeta_{-\beta}\varepsilon
(\alpha,-\alpha_i)=\zeta_\beta\varepsilon(\alpha,-\alpha_i)=
-\zeta_\beta\varepsilon(-\alpha_i,\alpha).\]
By Definition~\ref{defcan}(b), the left hand side equals $\zeta_\alpha c_i$.
By the first part of the proof and Definition~\ref{defcan}(a), the right 
hand side of the above identity equals $-\zeta_\beta \varepsilon(\alpha_i,
\beta)=-\zeta_\beta c_i$. Hence, we must have $\zeta_\alpha=-\zeta_\beta$. 
\end{proof}

\begin{prop} \label{propcan2} Let $\cL$ be a special Chevalley system as
above, with corresponding signs $\{c_i\mid i \in I\}$. Let $i \in I$.
Then the formulae in {\rm (a)} and {\rm (b)} of Definition~\ref{defcan} 
actually hold for all $\alpha\in \Phi$, that is, we have  
\begin{alignat*}{2}
[e_{\alpha_i},e_\alpha]&=c_i(q_{\alpha_i,\alpha}+1)e_{\alpha+\alpha_i} &&
\quad \mbox{for any $\alpha\in \Phi$, $\alpha+\alpha_i\in \Phi$},\\
[e_{-\alpha_i},e_\alpha]&=c_i(p_{\alpha_i,\alpha}+1)e_{\alpha-\alpha_i} 
&&\quad \mbox{for any $\alpha\in \Phi$, $\alpha-\alpha_i\in \Phi$}.
\end{alignat*}
Furthermore, $[e_\alpha,e_{-\alpha}]=-(-1)^{\hgt(\alpha)}h_\alpha$
for all $\alpha\in \Phi$.
\end{prop}

\begin{proof} Let $\alpha\in \Phi$ be such that $\beta:=\alpha-\alpha_i
\in \Phi$. Then 
\[ [e_{-\alpha_i},e_\alpha]=\varepsilon(-\alpha_i,\alpha)(q_{-\alpha_i,
\alpha}+1)e_{\beta}.\]
If $\alpha\in \Phi^-$, then the assertion is clear since $\varepsilon
(-\alpha_i,\alpha)=c_i$ by Definition~\ref{defcan}(b). Now let $\alpha
\in \Phi^+$; note that also $\beta\in \Phi^+$. By Lemma~\ref{lem21} and 
Definition~\ref{defcan}(a), we have $\varepsilon(-\alpha_i,\alpha)=
\varepsilon(\alpha_i,\beta)=c_i$, as required.

Next consider $\varepsilon(\alpha_i,\alpha)$ where $\alpha\in \Phi$ is 
such that $\alpha':=\alpha+\alpha_i\in \Phi$. If $\alpha\in \Phi^+$, then
the assertion is clear by Definition~\ref{defcan}(a). If $\alpha\in 
\Phi^-$, then  
\[\varepsilon(\alpha_i,\alpha)=\varepsilon(\alpha_i,\alpha'-\alpha_i)=
\varepsilon(-\alpha_i,\alpha')=c_i\]
as required, where the second equality holds by Lemma~\ref{lem21} and the 
third equality holds by Definition~\ref{defcan}(b).

Finally, consider the identity $[e_\alpha,e_{-\alpha}]=\zeta_\alpha 
h_\alpha$ for $\alpha\in \Phi$. We want to prove that $\zeta_\alpha=
-(-1)^{\hgt(\alpha)}$. Since $h_{\alpha}=-h_{-\alpha}$, it is sufficient to 
do this assuming that $\alpha\in \Phi^+$. Then we proceed by induction on 
$\hgt(\alpha)$. If $\hgt(\alpha)=1$, then $\alpha=\alpha_i$ for some 
$i\in I$ and $\zeta_{\alpha_i}=1$ holds by assumption. Now let 
$\hgt(\alpha)>1$. We can find some $i\in I$ such that $\beta:=\alpha-
\alpha_i\in \Phi^+$. But then, by Lemma~\ref{lem21} and induction, we 
have $\zeta_\alpha=-\zeta_\beta=(-1)^{\hgt(\beta)}=-(-1)^{\hgt(\alpha)}$, 
as claimed.
\end{proof}

\begin{rem} \label{thmcan0} Let $\cL=\{e_\alpha\mid \alpha\in \Phi\}$ be 
a special Chevalley system. Let us fix a sign $\epsilon\in \{\pm 1\}$. 
Then we obtain a new (not necessarily special) Chevalley system 
$\breve{\cL}^\epsilon=\{\breve{e}_\alpha^\epsilon\mid \alpha\in \Phi\}$ 
by setting
\begin{equation*}
\breve{e}_\alpha^\epsilon:=\left\{\begin{array}{rl} e_\alpha  & \qquad 
\mbox{if $\alpha\in \Phi^+$},\\ -\epsilon (-1)^{\hgt(\alpha)}e_\alpha & 
\qquad \mbox{if $\alpha \in \Phi^-$}. \end{array}\right. 
\end{equation*}
Consequently, for any $i\in I$ and $\alpha\in \Phi$, we have
\begin{alignat*}{2}
[\breve{e}_{\alpha_i}^\epsilon,\breve{e}_\alpha^\epsilon]&=
c_i(q_{\alpha_i,\alpha}+1)\breve{e}_{\alpha+\alpha_i}^\epsilon &&
\quad \mbox{if $\alpha\in \Phi^+$ and $\alpha+\alpha_i\in \Phi$},\\
[\breve{e}_{\alpha_i}^\epsilon,\breve{e}_\alpha^\epsilon]&=
-c_i(q_{\alpha_i,\alpha}+1)\breve{e}_{\alpha+\alpha_i}^\epsilon &&
\quad \mbox{if $\alpha\in \Phi^-$ and $\alpha+\alpha_i\in \Phi$},\\
[\breve{e}_{-\alpha_i}^\epsilon,\breve{e}_\alpha^\epsilon]&=
\epsilon c_i(p_{\alpha_i,\alpha}+1)\breve{e}_{\alpha-\alpha_i}^\epsilon
&&\quad \mbox{if $\alpha\in \Phi^+$ and $\alpha-\alpha_i\in \Phi$},\\
[\breve{e}_{-\alpha_i}^\epsilon,\breve{e}_\alpha^\epsilon]&=
-\epsilon c_i(p_{\alpha_i,\alpha}+1)\breve{e}_{\alpha-\alpha_i}^\epsilon
&&\quad \mbox{if $\alpha\in \Phi^-$ and $\alpha-\alpha_i\in \Phi$}.
\end{alignat*}
Furthermore, Proposition~\ref{propcan2} implies that 
$[\breve{e}_\alpha^\epsilon,\breve{e}_{-\alpha}^\epsilon]=\epsilon 
h_\alpha$ for all $\alpha\in \Phi$. Thus, if $\epsilon=1$, then 
$\breve{\cL}^+$ satisfies the standard normalisation of the root elements 
employed, for example, in Humphreys \cite[\S 8.3]{H}. On the other hand, 
if $\epsilon=-1$, then $\breve{\cL}^-$ corresponds to the normalisation in
Bourbaki \cite[Ch.~VIII, \S 2, no.~4]{bour2}. It is probably a matter of 
taste which normalisation is preferable in a given context. In any case, 
it is sufficient to determine the structure constants for either~$\cL$ or 
$\breve{\cL}^\epsilon$. We shall see that the Frenkel--Kac approach (to be 
discussed in Section~\ref{secflm}) is most directly related to~$\cL$.
\end{rem}

\begin{rem} \label{allsigns} Let $\cL$ be a special Chevalley system.
Let $\alpha,\beta\in \Phi$ be arbitrary such that $\alpha+\beta\in 
\Phi$. We have the following two identities:
\begin{align*}
\varepsilon(\alpha,\beta)& =\varepsilon(-\alpha,-\beta) \tag{a}\\
\varepsilon(\alpha,\beta)&=(-1)^{\hgt(\beta)}\varepsilon(\beta,
-\alpha-\beta)=(-1)^{\hgt(\alpha)}\varepsilon(-\alpha-\beta,\alpha). 
\tag{b}
\end{align*}
(The first one is contained in Proposition~\ref{propcan1}(b); the second 
one follows from Proposition~\ref{propcan2} and  Lemma~\ref{cass1}, 
applied with $\gamma_1:=\alpha$, $\gamma_2:=\beta$, $\gamma_3=
-\alpha-\beta$.) Assume that the signs $\varepsilon(\alpha,\beta)$ are 
known whenever $\alpha,\beta\in \Phi^+$. Then the above two identities
provide simple rules for computing the signs $\varepsilon(\alpha,\beta)$ 
for \textit{all} $\alpha,\beta\in \Phi$. 
\end{rem}

\begin{thm}[Lusztig \protect{\cite{L1,L2,L5}}] \label{thmcan} There 
exists a special Chevalley system $\cL=\{e_\alpha\mid \alpha\in \Phi\}$
for $\fg$.
\end{thm}

\begin{proof} All the work has been done in \cite{L1,L2,L5} but the
formulation there is a bit different; so we need to say a few words about 
the translation from \cite{L1,L2,L5} to the present setting. For each 
$i\in I$ let us choose some elements $e_i\in \fg_{\alpha_i}$ and $f_i\in
\fg_{-\alpha_i}$ such that $h_i=[e_i,f_i]$. Then, by \cite{L1,L2,L5}, 
there exists a collection of root elements $\cL=\{\be_\alpha^+\mid 
\alpha\in \Phi\}$ such that the following relations hold.
\begin{alignat*}{2}
[f_i,\be_{\alpha_i}^+]&=[e_i,\be_{-\alpha_i}] && \quad \mbox{for all
$i\in I$},\tag{L1}\\
[e_i,\be_\alpha^+]&=(q_{\alpha_i,\alpha}+1)\be_{\alpha+\alpha_i}^+ &&\quad 
\mbox{if $i\in I$, $\alpha\in \Phi$, $\alpha+\alpha_i\in \Phi$}, \tag{L2}
\\ [f_i,\be_\alpha^+]&=(p_{\alpha_i,\alpha}+1) \be_{\alpha-\alpha_i}^+ 
&& \quad \mbox{if $i\in I$, $\alpha\in \Phi$, $\alpha-\alpha_i\in \Phi$}. 
\tag{L3}
\end{alignat*}
(See also the exposition in \cite[\S 2.7]{mynotes}.) Such a collection is 
unique up to replacing each $\be_\alpha^+$ by $\xi \be_\alpha^+$, for some 
$0\neq \xi \in \C$. We can certainly write $\be_{\alpha_i}^+=c_i^+e_i$ and 
$\be_{-\alpha_i}^+=c_i^-f_i$ for $i\in I$, where $0\neq c_i^\pm \in \C$. 
Using (L1) one immediately sees that $c_i^-=-c_i^+$ for $i\in I$. So we have 
\[ \be_{\alpha_i}^+=c_i^+e_i \quad\mbox{and}\quad \be_{-\alpha_i}^+
=-c_i^+f_i \qquad \mbox{for all $i\in I$}.\]
Furthermore, let $i\neq j$ in $I$ be such that $a_{ij}\neq 0$. Arguing as 
in Remark~\ref{remcan1}(b), one sees that $c_j^+=-c_i^+$. Since the graph
of $A$ is connected it follows that, if we fix some $i_0\in I$, then 
$c_i^+=\pm c_{i_0}^+$ for all $i\in I$. Hence, replacing each 
$\be_\alpha^+$ by $c_{i_0}^{-1}\be_\alpha^+$ we obtain a new collection 
of root elements which still satisfy (L1), (L2), (L3), but where now we 
have $c_i^+=\pm 1$. Thus, we may assume without loss of generality that 
$c_i^+=\pm 1$ for all $i\in I$. It then also follows that $\{\be_\alpha^+
\mid \alpha\in \Phi\}$ is a Chevalley system, that is, ($\clubsuit$) 
holds; see \cite[Theorem~5.7]{mylie} or \cite[Theorem~2.7.16]{mynotes} 
for a detailed argument. Now set
\begin{equation*}
e_\alpha:=\left\{\begin{array}{rl} \be_\alpha^+  & \quad \mbox{if
$\alpha\in \Phi^+$},\\ -\be_\alpha^+ & \quad \mbox{if $\alpha \in \Phi^-$}; 
\end{array}\right. \tag{$*$}
\end{equation*}
in particular, $e_{\alpha_i}=\be_{\alpha_i}^+=c_i^+e_i$ and $e_{-\alpha_i}
=-\be_{-\alpha_i}^+=c_i^+ f_i$ for $i\in I$. First of all, this
implies that $[e_{\alpha_i},e_{-\alpha_i}]=[e_i,f_i]=h_i$ for $i\in I$.
Secondly, (L2) and (L3) translate into the two identities in 
Proposition~\ref{propcan2}. 
\end{proof}

In the following two sections, we address the problem of computing the
signs $\varepsilon(\alpha,\beta)$ in the structure constants for a special 
Chevalley system of~$\fg$. By Proposition~\ref{cancor}, these signs are 
uniquely determined once we fix one of the two possible collections of 
signs $\{c_i\mid i \in I\}$ satisfying the condition in 
Remark~\ref{remcan1}(b). Note that, by ($*$) in the above proof, the 
signs $\varepsilon(\alpha,\beta)$ also describe the structure constants 
for the positive part of Lusztig's canonical basis.

\section{The construction of Frenkel and Kac} \label{secflm}

We keep the general setting of the previous section, but now we assume
that the root system $\Phi$ is simply laced or, equivalently, that $A$ is 
symmetric. Then all roots in $\Phi$ have the same length and we can 
normalise $\langle \;,\;\rangle$ such that
\begin{center}
\fbox{$\langle \alpha,\alpha\rangle=2\qquad \mbox{for all 
$\alpha\in \Phi$}$.}
\end{center} 
Consequently, the Cartan matrix $A$ is the Gram matrix of $\langle \;,
\;\rangle$ with respect to the basis $\{\alpha_i\mid i \in I\}$ of~$E$. 
Here are some further consequences of the above assumption.

\begin{rem} \label{remflm1} Let $\alpha,\beta\in \Phi$ be such that 
$\alpha+\beta\in \Phi$. Since $\langle \gamma,\gamma\rangle=2$ for all 
$\gamma\in \Phi$, we conclude that $\langle \alpha,\beta\rangle=-1$. 
Similarly, if $\beta-\alpha\in \Phi$, then $\langle \alpha,\beta\rangle=1$. 
Hence, if $\alpha+\beta\in \Phi$, then $\beta-\alpha \not\in \Phi$ and
so $q_{\alpha,\beta}=0$; furthermore, $p_{\alpha,\beta}=1$ since 
$-1=\langle \alpha,\beta\rangle=q_{\alpha, \beta}-p_{\alpha,\beta}$.
\end{rem}

The following discussion is based on Frenkel--Lepowsky--Meurman 
\cite[Chap.~6]{FLM}. Let $Q:=\Z\Phi\subseteq E$ be the lattice spanned by 
$\Phi$ and 
\[\varepsilon \colon Q\times Q \rightarrow \{\pm 1\}\]
be a function satisfying the following conditions, for all 
$\alpha,\beta, \gamma\in Q$:
\begin{gather*}
\varepsilon(\alpha,\beta)\varepsilon(\alpha+\beta,\gamma) =
\varepsilon(\beta,\gamma)\varepsilon(\alpha,\beta+\gamma),\tag{FLM1}\\
\varepsilon(\alpha,0)=\varepsilon(0,\alpha)=1,\tag{FLM2}\\
\varepsilon(\alpha,\beta)\varepsilon(\beta,\alpha)=(-1)^{\langle
\alpha,\beta\rangle}.\tag{FLM3}
\end{gather*}

\begin{rem} \label{remflm2} Assume that $\varepsilon\colon Q\times Q 
\rightarrow \{\pm 1\}$ is a bi-multiplicative function, that is, we have 
\begin{align*}
\begin{array}{c} \varepsilon(\alpha+\beta,\gamma)=\varepsilon(\alpha,\gamma)
\varepsilon (\beta,\gamma)\\ \varepsilon(\alpha,\beta+\gamma)=
\varepsilon(\alpha,\beta)\varepsilon(\alpha,\gamma)\end{array}
\qquad\mbox{for all $\alpha,\beta,\gamma\in \Phi$}.
\end{align*}
Then (FLM1), (FLM2) immediately follow. And we automatically have 
\[\varepsilon(\alpha,-\beta)=\varepsilon(\alpha,\beta)=\varepsilon(-\alpha,
\beta) \qquad \mbox{for all $\alpha,\beta \in Q$}. \]
In all our examples below, the above conditions will be satisfied. 
\end{rem}

Let $\fg$ be a vector space over $\C$ with a basis $\{h_i\mid i \in I\}\cup 
\{e_\alpha\mid\alpha\in \Phi\}$. Let $\fh:=\langle h_i\mid i \in I\rangle_\C
\subseteq \fg$. For $\alpha\in \Phi$ we write $\alpha=\sum_{i \in I} 
n_i\alpha_i$ where $n_i\in \Z$ for all $i\in I$; then we also set 
$h_\alpha:=\sum_{i \in I} n_ih_i\in \fh$. Define a bilinear product 
$[\;,\;]\colon \fg\times \fg\rightarrow \C$ by the following rules,
where $i,j\in I$ and $\alpha,\beta \in \Phi$:
\begin{align*}
\left[h_i,h_j\right]&=0,\\ \left[h_i,e_\alpha\right]&=
-\left[e_\alpha,h_i\right]=\langle\alpha_i,\alpha\rangle e_\alpha,\\
\left[e_\alpha,e_{-\alpha}\right]&=\varepsilon(\alpha,-\alpha)h_\alpha,\\
\left[e_\alpha,e_\beta\right]&=0 \qquad \mbox{if $\alpha+\beta\not\in\Phi
\cup\{\underline{0}\}$},\\
\left[e_\alpha,e_\beta\right]&=\varepsilon(\alpha,\beta)e_{\alpha+\beta} 
\qquad \mbox{if $\alpha+\beta\in\Phi$}.
\end{align*}
Now we can state:

\begin{thm}[Frenkel--Lepowsky--Meurman \protect{\cite[\S 6.2]{FLM}}] 
\label{thmflm} Recall that the function $\varepsilon\colon Q\times Q
\rightarrow\{\pm 1\}$ is assumed to satisfy {\rm (FLM1)--(FLM3)}. Equipped 
with the above product, $\fg$ is a simple Lie algebra, $\fh\subseteq \fg$
is a Cartan subalgebra and the collection $\cL=\{e_\alpha\mid \alpha\in 
\Phi\}$ is a Chevalley system for~$\fg$, with corresponding sign
function $\varepsilon(\alpha,\beta)$.
\end{thm}

This result puts the original construction of Frenkel--Kac \cite[\S 
2.3]{FK} in a slightly more general, axiomatic setting. The proof is a 
completely elementary verification where the crucial thing is to check that 
$[\;,\;]\colon \fg\times \fg \rightarrow \C$ satisfies the Jacobi identity. 
This is done in \cite[Theorem~6.2.1]{FLM}; see also Kac \cite[\S 7.8]{kac}
and Springer \cite[\S 10.2]{Sp}. A very detailed exposition of the argument, 
explicitly working out the various cases that have to be considered, is 
contained in De Graaf \cite[\S 5.13]{graaf}. Note, however, that \cite{graaf},
\cite{kac}, \cite{Sp} assume that $\varepsilon(\alpha,\alpha)=-1$ for all 
$\alpha\in \Phi$ (see \cite[(6.11)]{graaf}, \cite[(7.8.2)]{kac}, 
\cite[10.2.3]{Sp}); so one has to go through their arguments and check that 
everything works assuming only that (FLM1)--(FLM3) hold. Once this is done, 
the fact that $\fg$ is simple and $\fh$ is a Cartan subalgebra follows quite 
easily. We refer to \cite{graaf}, \cite{FLM}, \cite{kac}, \cite{Sp} for 
further details.

\begin{exmp} \label{expflm1} Choose any orientation of the incidence 
graph of $A$ (see Remark~\ref{graphA}). Following Kac \cite[\S 7.8]{kac}, 
for $i,j\in I$ we define 
\[ \varepsilon(\alpha_i,\alpha_j):=\left\{\begin{array}{rl} -1 & \quad 
\begin{array}{l} \mbox{if $i=j$},\end{array}\\ -1 & \quad \begin{array}{l} 
\mbox{if $i\neq j$ are joined by an edge and} \\ \mbox{the arrow points 
from $i$ to $j$}, \end{array}\\ 1 & \quad \begin{array}{l}\mbox{otherwise}. 
\end{array} \end{array}\right.\]
Given arbitrary $\alpha,\beta\in Q$, we write $\alpha=\sum_{i\in I} n_i
\alpha_i$ and $\beta=\sum_{j\in I}m_j\alpha_j$ where $n_i,m_j\in \Z$; then 
we set $\varepsilon(\alpha,\beta):=\prod_{i,j\in I} \varepsilon(\alpha_i,
\alpha_j)^{n_i m_j}$. Thus, we have defined a bi-multiplicative function 
$\varepsilon\colon Q \times Q\rightarrow \{\pm 1\}$; in particular,
(FLM1) and (FLM2) hold by Remark~\ref{remflm2}. We also have 
\begin{align*}
\varepsilon(\alpha,\alpha)=(-1)^{\frac{1}{2}\langle \alpha,\alpha\rangle}
\qquad\mbox{for all  $\alpha\in Q$};\tag{FLM3$^*$}
\end{align*}
see \cite[Lemma~5.13.3]{graaf} for a detailed proof. Applying this to 
$\alpha+\beta$ for $\alpha,\beta \in Q$, one sees that (FLM3) holds. So we 
have a corresponding simple Lie algebra~$\fg$ as in Theorem~\ref{thmflm}.
Using (FLM3$^*$) we have $\varepsilon(\alpha,-\alpha)=\varepsilon(\alpha,
\alpha)=(-1)^{\frac{1}{2}\langle \alpha,\alpha\rangle}=-1$ and, hence, 
\begin{equation*}
[e_\alpha,e_{-\alpha}]=-h_\alpha \qquad\mbox{for all $\alpha\in \Phi$}.
\tag{$\dagger$}
\end{equation*}
This is exactly the convention for the root elements $\cL=\{e_\alpha 
\mid \alpha \in \Phi\}$ that is used in Bourbaki \cite[Ch.~VIII, \S 2, 
no.~4]{bour2}. Comparing ($\dagger$) and the formula in 
Proposition~\ref{propcan2}, we see that $\cL$ is definitely not a 
``special'' Chevalley system in the sense of Definition~\ref{defcan}.
\end{exmp}

\begin{rem} \label{rem2} 
Let $\varepsilon\colon Q\times Q\rightarrow\{\pm 1\}$ be as in 
Example~\ref{expflm1}. If, instead of ($\dagger$), one prefers to work with 
the relation $[e_\alpha,e_{-\alpha}]=h_\alpha$ for all $\alpha\in \Phi$ (which 
is the standard convention, for example, in Humphreys \cite{H}), then one can 
proceed as follows. We define a new collection $\{\hat{e}_\alpha
\mid \alpha\in \Phi\}$ of root elements by 
\[\hat{e}_\alpha:=e_\alpha \qquad \mbox{and}\qquad \hat{e}_{-\alpha} 
:=-e_{-\alpha} \qquad \mbox{for all $\alpha\in \Phi^+$}.\]
Then we certainly have $[\hat{e}_{\alpha},\hat{e}_{-\alpha}]=h_\alpha$ for all
$\alpha\in \Phi$. Furthermore, for any $\alpha\in \Phi$, we set $\mbox{sgn}
(\alpha):=1$ if $\alpha\in \Phi^+$ and $\mbox{sgn}(\alpha):=-1$ if $\alpha
\in \Phi^-$. One easily checks that
\[ [\hat{e}_{\alpha},\hat{e}_{\beta}]=\mbox{sgn}(\alpha)\mbox{sgn}(\beta)
\mbox{sgn}(\alpha+\beta)\varepsilon(\alpha,\beta)\hat{e}_{\alpha+\beta}\]
for all $\alpha,\beta\in \Phi$ such that $\alpha+\beta\in \Phi$.
\end{rem}

\begin{rem} \label{othere} There exist choices of $\varepsilon \colon 
Q\times Q\rightarrow \{\pm 1\}$ which satisfy (FLM1)--(FLM3) but which
are not related to an orientation of the graph of~$A$. See, e.g., 
Cohen--Griess--Lisser \cite[\S 2]{CGL} for such an example for~$\Phi$ 
of type~$E_8$, where we still have $\varepsilon(\alpha_i,\alpha_i)=-1$ 
for $i\in I$. (I~learned about this reference from Vavilov \cite{vav2}.)
\end{rem}

In order to establish a connection with Lusztig's canonical basis, we
will now choose a particular orientation of the incidence graph of~$A$, 
and we will slightly change the definition of $\varepsilon \colon Q 
\times Q\rightarrow \{\pm 1\}$ in Example~\ref{expflm1}. The following 
discussion is based on \cite[\S 3]{GeLa}, which itself is based on the 
Master's thesis of Lang \cite{Lang}.

\begin{exmp} \label{expflm2} We fix (one of the two possible) orientations 
of the graph of~$A$ where each vertex is either a sink or a source; let 
$\{c_i\mid i \in I\}$ be the collection of signs characterising that 
orientation. (See again Remark~\ref{graphA}.) Now we define for 
$i,j\in I$:
\[\varepsilon_0(\alpha_i,\alpha_j):=c_i^{a_{ij}}=\left\{\begin{array}{rl} 
1 & \; \begin{array}{l} \mbox{if $i=j$},\end{array}\\ -1 & \; 
\begin{array}{l} \mbox{if $i\neq j$ are joined by an edge and} \\ 
\mbox{the arrow points from $i$ to $j$}, \end{array}\\ 1 & \; 
\begin{array}{l} \mbox{otherwise}.\end{array}\end{array}\right.\]
(For the second equality note that $i\neq j$ are joined by an edge if and 
only if $a_{ij}\neq 0$; furthermore, $a_{ii}=2$.) Again, we extend this
bi-multiplicatively to all of~$Q$. Concretely, for arbitrary $\alpha,
\beta \in Q$, we write $\alpha=\sum_{i\in I} n_i \alpha_i$ and 
$\beta=\sum_{j\in I}m_j\alpha_j$ where $n_i,m_j\in \Z$; then 
\begin{equation*}
\varepsilon_0(\alpha,\beta)=\prod_{i,j\in I} c_i^{a_{ij}n_im_j}=
\prod_{i \in I} c_i^{n_i\langle \alpha_i,\beta\rangle} \qquad 
\mbox{(see \cite[Def.~3.5]{GeLa})}.\tag{$\spadesuit$}
\end{equation*}
The following result provides, at the same time, a new proof of the formula 
for the structure constants in \cite[Theorem~3.9]{GeLa}, and a new proof 
of Theorem~\ref{thmcan}, for $\fg$ with a simply laced root system~$\Phi$.
\end{exmp}

\begin{prop}  \label{propGL} The above function $\varepsilon_0 \colon Q
\times Q\rightarrow \{\pm 1\}$ satisfies {\rm (FLM1)--(FLM3)} and, hence, 
we obtain a corresponding simple Lie algebra $\fg$ as in 
Theorem~\ref{thmflm}. The collection $\cL=\{e_\alpha\mid \alpha\in 
\Phi\}$ is a special Chevalley system, where the corresponding signs (as 
in Definition~\ref{defcan}) are exactly the signs $\{c_i\}$ in 
Example~\ref{expflm2}. We have
\[ \varepsilon_0(\alpha,\alpha)=\varepsilon_0(\alpha,-\alpha)=
-(-1)^{\hgt(\alpha)} \qquad \mbox{for all $\alpha\in \Phi$}.\] 
\end{prop}

\begin{proof} By construction, $\varepsilon_0$ is bi-multiplicative; hence,
(FLM1) and  (FLM2) hold; see Remark~\ref{remflm2}. By 
\cite[Lemma~3.7]{GeLa} (and its proof), we know that (FLM3) also holds. 
Hence, Theorem~\ref{thmflm} applies, which yields the existence of $\fg$ 
and shows that $\cL=\{e_\alpha\mid \alpha\in \Phi\}$ is a Chevalley 
system. Now let $i\in I$. Then $[e_{\alpha_i}, e_{-\alpha_i}]=h_i$ since 
$\varepsilon_0(\alpha_i,-\alpha_i)=\varepsilon_0(\alpha_i, \alpha_i)=1$. 
Furthermore, let $\alpha\in \Phi$ be such that $\alpha\pm \alpha_i\in 
\Phi$. Then $\langle \alpha_i,\alpha\rangle=\pm 1$ by Remark~\ref{remflm1}. 
Hence, by ($\spadesuit$) and the bi-multiplicativity of~$\varepsilon_0$, we 
have 
\[\varepsilon_0(\pm \alpha_i,\alpha)=\varepsilon_0(\alpha_i,\alpha)=
c_i^{\langle\alpha_i,\alpha\rangle}=c_i\]
and so $[e_{\pm \alpha_i},e_{\alpha}]=c_ie_{\alpha\pm \alpha_i}$, as 
required. Thus, we do have a special Chevalley system $\cL=\{e_\alpha
\mid \alpha\in \Phi\}$ as in Definition~\ref{defcan}. Consequently, 
the statement concerning $\varepsilon_0(\alpha,\alpha)$ is clear by 
Proposition~\ref{propcan2}.
\end{proof}

\begin{rem} \label{vecA} Having fixed a special orientation as in 
Example~\ref{expflm2}, let $\vQ(A)$ be the set of all pairs
$(i,j)\in I \times I$ such that  
\[ i\neq j, \qquad a_{ij}\neq 0 \qquad \mbox{and} \qquad c_i=-1.\]
Thus, $\vQ(A)$ is the set of all oriented edges in the graph of $A$, 
where the orientation is expressed using the signs $\{c_i\mid i \in I\}$. 
Let us set 
\[ \vec{\rho}(\alpha,\beta):=-\sum_{(i,j)\in \vQ(A)} n_im_j
\qquad \mbox{for $\alpha,\beta\in Q$}\]
where we write, as usual, $\alpha=\sum_{i \in I} n_i\alpha_i$ and $\beta
=\sum_{j \in I} m_j\alpha_j$ with $n_i,m_j\in \Z$. Then the above function 
$\varepsilon_0\colon Q \times Q\rightarrow \{\pm 1\}$ is more concisely 
(and more efficiently from a computational point of view) given by the 
formula $\varepsilon_0(\alpha,\beta)=(-1)^{\vec{\rho}(\alpha,\beta)}$.
\end{rem}

\begin{cor} \label{vecA1} Recall that $\Phi$ is assumed to be simply laced. 
Let $\varepsilon_0\colon Q \times Q\rightarrow \{\pm 1\}$ be as in 
Example~\ref{expflm2}. Then there is a special Chevalley system $\cL=
\{e_\alpha\mid \alpha\in \Phi\}$ in the corresponding Lie algebra $\fg$ 
such that
\[ [e_\alpha,e_\beta]=(-1)^{\vec{\rho}(\alpha,\beta)}(q_{\alpha,\beta}+1)
e_{\alpha+\beta} \qquad \mbox{if $\alpha,\beta,\alpha+\beta\in \Phi$}.\]
\end{cor}

\begin{proof} This is just a re-formulation of Proposition~\ref{propGL},
using the notation in Remark~\ref{vecA}.
\end{proof}

\begin{exmp} \label{expAn} Assume that $\Phi$ is of type $A_{r}$ for $r\geq 
1$. We let $I=\{1,\ldots,r\}$ where the notation is such that $i,i+1$ are 
joined by an edge for $1\leq i \leq r-1$ (as in the diagram below). We 
choose an orientation such that $i\in I$ is a source if $i$ is even, and
$i$ is a sink if $i$ is odd. Thus, $c_1=c_3=\ldots =1$ and $c_2=c_4=\ldots 
=-1$. We indicate this in the graph of $A$ as follows:
\begin{center}
\begin{picture}(200,30) 
\put(  5,15){$A_r$}
\put(  4, 4){$\scriptstyle{r \geq 1}$}
\put( 40,10){\circle*{5}}
\put( 37,16){$1^+$}
\put( 41,10){\line(1,0){28}}
\put( 70,10){\circle*{5}}
\put( 67,16){$2^-$}
\put( 70,10){\line(1,0){20}}
\put(100,10){\circle*{1}}
\put(110,10){\circle*{1}}
\put(120,10){\circle*{1}}
\put(130,10){\line(1,0){20}}
\put(150,10){\circle*{5}}
\put(138,16){$r{-}1^\pm$}
\put(150,10){\line(1,0){30}}
\put(180,10){\circle*{5}}
\put(178,16){$r^\mp$}
\end{picture}
\end{center}
So we have $\vQ(A)=\{(2,1),(2,3), (4,3),(4,5),(6,5),\ldots\}$. Now it 
is well known that 
\[ \Phi^+=\bigl\{\alpha_{ij}:=\alpha_{i}+\alpha_{i+1}+\ldots +
\alpha_{j-1} \mid 1 \leq i<j\leq r+1\bigr\}.\]
Let $1\leq i<j \leq r+1$ and $1\leq k<l\leq r+1$. The above description 
shows that $\alpha_{ij} +\alpha_{kl}\in \Phi$ if and only if either $k=j$ 
or $l=i$, in which case we have $\alpha_{ij}+\alpha_{jl}=\alpha_{ij}$ or 
$\alpha_{ij}+\alpha_{ki}=\alpha_{kj}$, respectively. A straightforward 
computation shows that 
\begin{equation*}
\vec{\rho}(\alpha_{ij},\alpha_{kl})= \left\{\begin{array}{rl}
0 & \quad \mbox{if $k=j$ is even or $l=i$ is odd},\\ -1 & \quad 
\mbox{if $k=j$ is odd or $l=i$ is even}.  \end{array}\right.\tag{a}
\end{equation*}
(Note that there is at most one pair in $\vQ(A)$ which gives a 
non-zero contribution to the sum defining $\vec{\rho}(\alpha_{ij},
\alpha_{kl})$ if $\alpha_{ij}+\alpha_{kl}\in \Phi$.)

Now (a) readily implies the following formulation, which is similar to 
that in Rylands \cite[3.1]{Ry}. If $\alpha,\beta\in\Phi^+$ are such that
$\alpha+\beta\in \Phi$, then 
\begin{gather*}
\vec{\rho}(\alpha,\beta)-\vec{\rho} (\beta,\alpha)\in \{\pm 1\} \quad
\mbox{and}\qquad\tag{b}\\
\varepsilon_0(\alpha,\beta)=\left\{\begin{array}{rl} 1 & \quad
\mbox{if $\vec{\rho}(\alpha,\beta)>\vec{\rho}(\beta,\alpha)$},\\ -1 & 
\quad \mbox{if $\vec{\rho}(\alpha,\beta)<\vec{\rho}(\beta,\alpha)$}. 
\end{array}\right.\tag{c}
\end{gather*}
This will turn out to be useful in the proof of Proposition~\ref{prop41} 
below.
\end{exmp}

\begin{rem} \label{remL2} We may also define $\varepsilon \colon Q\times Q
\rightarrow \{\pm 1\}$ in Example~\ref{expflm1} using one of the two 
special orientations of the graph of $A$ as in Example~\ref{expflm2}. 
Since the only difference between $\varepsilon$ and $\varepsilon_0$ occurs
in the definition of the values on pairs of basis elements $(\alpha_i,
\alpha_i)$ for $i\in I$, one immediately sees that 
\[ \varepsilon(\alpha,\beta)=\varepsilon_0(\alpha,\beta)\prod_{i\in I}
(-1)^{n_im_i}\]
where we write $\alpha=\sum_{i \in I} n_i\alpha_i$ and $\beta=
\sum_{i \in I} m_i\alpha_i$ with $n_i,m_i\in \Z$. 
\end{rem}

\section{Non simply laced root systems} \label{secnon}

Let $\Phi\subseteq E$ and $\langle \;,\;\rangle\colon E\times E\rightarrow 
\R$ be as before, but let us now drop the assumption that $\Phi$ be simply 
laced. Thus, $\Phi$ may also be of type $B_r$, $C_r$ ($r\geq 2$), $G_2$ or 
$F_4$. Let $\fg$ be a simple Lie algebra over $\C$ with root system $\Phi$. 
It is well-known (see, e.g., De Graaf \cite[\S 5.15]{graaf} or Kac 
\cite[\S 7.9]{kac}) that $\fg$ can be constructed using a ``folding'' 
procedure from a suitable simple Lie algebra $\fg^\circ$ with a simply 
laced root system. (We will have $\fg^\circ=\fg$ if $\Phi$ itself is 
simply laced.) We do not need to go into the exact details of the 
construction; let us just summarise the main points that will be important
to us here.

First we define an integer $e\geq 1$ as follows. We set $e=1$ if $\Phi$ is 
simply laced. If $\Phi$ is of type $B_r$, $C_r$ or $F_4$, then $e=2$; if
$\Phi$ is of type~$G_2$, then $e=3$. Now, given $\Phi$ and $\{\alpha_i
\mid i \in I\}$, there exists a pair $(\Phi^\circ,\tau)$, where 
$\Phi^\circ$ is a simply laced root system and $\tau\colon \Phi^\circ 
\rightarrow \Phi^\circ$ is a permutation of order $e\geq 1$, such that 
the following conditions hold.
\begin{itemize}
\item[(F1)] There is a linear map $\eta\colon \Z\Phi^\circ \rightarrow 
\Z\Phi$ such that $\eta(\Phi^\circ)=\Phi$. 
\item[(F2)] The fibres of $\eta|_{\Phi^\circ}\colon \Phi^\circ \rightarrow 
\Phi$ are exactly the orbits of $\tau$ on $\Phi^\circ$.
\item[(F3)] $\Pi^\circ:=\eta^{-1}(\{\alpha_i \mid i \in I\})$ is
a system of simple roots for $\Phi^\circ$.
\end{itemize}
This easily follows from the consideration of the various cases; see, 
e.g., \cite[\S 13.3]{C1}, \cite[Table~2]{GeLa}. See also Lusztig 
\cite[14.1.4--14.1.6]{L10} where ``folding'' is treated in a considerably 
more general setting. 

Now consider a simple Lie algebra $\fg^\circ$ with root system $\Phi^\circ$.
We assume that $\fg^\circ$ is constructed as in Proposition~\ref{propGL},
using a function $\varepsilon_0^\circ \colon\Z\Phi^\circ\times\Z\Phi^\circ
\rightarrow \{\pm 1\}$ as in Example~\ref{expflm2}. Such a function 
satisfies a further condition, as follows. Let $\alpha,\beta\in \Phi$ be 
such that $\alpha+\beta \in \Phi$. Let $S(\alpha,\beta)$ be the set of 
all $\alpha^\circ, \beta^\circ\in \Phi^\circ$ such that 
$\eta(\alpha^\circ)=\alpha$, $\eta(\beta^\circ)=\beta$ and $\alpha^\circ+
\beta^\circ\in\Phi^\circ$. Then, by \cite[Lemma~4.9]{GeLa}, we have:
\begin{itemize}
\item[(F4)] Invariance: $\quad S(\alpha,\beta)\neq \varnothing$ and 
$\varepsilon^\circ$ is constant on $S(\alpha,\beta)$.
\end{itemize}
The above function $\varepsilon_0^\circ$ is defined by an orientation of 
the graph of the Cartan matrix of $\Phi^\circ$ for which every vertex is 
either a sink or a source. Since we did not introduce a notation for an 
index set of $\Pi^\circ$ as in (F3), we index the corresponding signs 
directly by $\Pi^\circ$ and denote them as 
\[ \{c_{\alpha^\circ}\mid\alpha^\circ \in\Pi^\circ\}.\]  
Now $\tau\colon \Phi^\circ\rightarrow\Phi^\circ$ induces a Lie algebra 
automorphism $\tilde{\tau}\colon\fg^\circ\rightarrow\fg^\circ$ such that 
$\fg=\{x\in \fg^\circ\mid \tilde{\tau}(x)=x\}$ is a simple Lie 
algebra with root system~$\Phi$; see \cite[\S 5.15]{graaf} or 
\cite[\S 7.9]{kac} for details. Futhermore, we have:

\begin{thm}[See \protect{\cite[Theorem~4.10]{GeLa}}] \label{thmF1}
In the above setting, there is a special Chevalley system $\cL=\{e_\alpha 
\mid \alpha\in \Phi\}$ for $\fg$, with the following properties. 
\begin{itemize}
\item[(a)] The corresponding signs $\{c_i\mid i \in I\}$ (as in 
Definition~\ref{defcan}) are given by $c_i= c_{\alpha^\circ}$ where 
$\alpha^\circ\in \Pi^\circ$ is such that $\eta(\alpha^\circ)=\alpha_i$.
\item[(b)] Let $\alpha, \beta \in \Phi$ be such that $\alpha+\beta 
\in \Phi$ and write 
\[\qquad [e_\alpha,e_\beta]=\varepsilon(\alpha,\beta)(q_{\alpha,\beta}+1) 
e_{\alpha+\beta}\quad\mbox{where $\varepsilon(\alpha,\beta)=\pm 1$}.\]
Then we have $\varepsilon(\alpha,\beta)=\varepsilon_0^\circ(\alpha^\circ,
\beta^\circ)$ where $(\alpha^\circ,\beta^\circ)\in S(\alpha,\beta)$.
\end{itemize}
\end{thm}

\begin{proof} In \cite{GeLa} this is proved using Lusztig's canonical 
basis for $\fg^\circ$. So one just needs to apply the translation from 
that basis to a special Chevalley system already employed in the proof 
of Theorem~\ref{thmcan}.
\end{proof}

In particular, the above result yields formulae for the structure constants
of~$\fg$ in terms of the structure constants of~$\fg^\circ$. It would also 
be desirable to find explicit formulae for the structure constants of~$\fg$ 
directly in terms of the roots in $\Phi$, without reference to $\Phi^\circ$. 
This is what we will do in this and the following section. 

The analogous problem was also considered by Ringel \cite{Ri} and, more 
recently, by Rylands \cite{Ry}, but for $\fg^\circ$ defined with respect to 
functions $\varepsilon^\circ\colon\Z\Phi^\circ\times\Z\Phi^\circ\rightarrow
\{\pm 1\}$ as in Example~\ref{expflm1} (which do not lead to special 
Chevalley systems). Nevertheless, we can adapt certain definitions and 
arguments from \cite{Ri}, \cite{Ry} to the present context, most notably 
the following one.

\begin{defn}[Cf.\ Ringel \protect{\cite[p.~139]{Ri}}] \label{defR}
Let $\vQ(A)\subseteq I \times I$ be defined as in Remark~\ref{vecA}.
Furthermore, we set $d_i:=1$ if $\alpha_i$ is a short root and $d_i:=e$ if
$\alpha_i$ is a long root. Thus, we have 
\[ d_ia_{ij}=a_{ji}d_j\qquad \mbox{for all $i,j\in I$}.\]
Let $\alpha,\beta \in Q=\Z\Phi$ and write $\alpha=\sum_{i \in I} n_i
\alpha_i$ and $\beta=\sum_{j \in I} m_j\alpha_j$ where $n_i,m_j\in \Z$. 
Then we define 
\[ \vec{\rho}(\alpha,\beta):=\sum_{(i,j)\in\vQ(A)} d_ia_{ij} 
n_im_j\quad\in\Z.\]
Note that, if $\Phi$ is simply laced, then $d_i=1$ and $a_{ij}=-1$
for all $(i,j)\in \vQ(A)$. Hence, in this case, the above definition 
of $\vec{\rho} (\alpha,\beta)$ reduces to that in Remark~\ref{vecA}.
\end{defn}

Note that the above definition of $\vec{\rho}(\alpha,\beta)$ is very 
similar to, but not exactly the same as that of $(\alpha,\beta)_\Omega$ 
in Ringel \cite[p.~139]{Ri}, since our $\vQ(A)$ only contains pairs 
$(i,j) \in I\times I$ with $i\neq j$. 

\begin{rem} \label{lemRi} 
The function $\vec{\rho}\colon Q\times Q \rightarrow \Z$ has the 
following symmetry property. Let $\vQp(A)$ and $\vec{\rho'}(\alpha,
\beta)$ be defined with respect to the collection of signs $\{-c_i\mid i 
\in I\}$ (see Remark~\ref{remcan1}). Then we certainly have 
$\vQp(A)=\{(j,i)\mid (i,j)\in \vQ(A)\}$ and so 
\[\vec{\rho'}(\alpha,\beta)=\sum_{(j,i)\in \vQ'(A)} (d_ja_{ji})n_jm_i=
\sum_{(i,j)\in \vQ(A)} (d_ia_{ij})m_in_j=\vec{\rho}(\beta,\alpha)\]
for all $\alpha,\beta \in Q$.
\end{rem}

\begin{exmp} \label{expG2} Let $\Phi$ be of type $G_2$. Let $I=\{1,2\}$ 
where the notation is such that $\alpha_1$ is long and $\alpha_2$ is short. 
Then $\vQ(A)=\{(2,1)\}$ if $c_1=1$ and $\vQ(A)=\{(1,2)\}$ if 
$c_1=-1$. Hence, if $\alpha=n_1\alpha_1+n_2\alpha_2\in Q$ and $\beta=m_1
\alpha_1+m_2\alpha_2\in Q$, then 
\[\vec{\rho}(\alpha,\beta)=\left\{\begin{array}{cl} -3n_2m_1 & \quad
\mbox{if $c_1=1$, $c_2=-1$},\\ -3n_1m_2 & \quad \mbox{if $c_1=-1$, 
$c_2=1$}. \end{array} \right.\]
(Note that $d_1=3$, $d_2=1$, $a_{12}=-1$ and $a_{21}=-3$.) Now consider
\[ \Phi^+=\{\alpha_1,\alpha_2,\alpha_1+\alpha_2,\alpha_1+2\alpha_2,
\alpha_1+3\alpha_2,2\alpha_1+3\alpha_2\}.\]
Let $\alpha,\beta\in \Phi$ be arbitrary such that $\alpha+\beta\in\Phi$. 
Then we claim that 
\begin{equation*}
[e_\alpha,e_\beta]=(-1)^{\vec{\rho}(\alpha,\beta)}(q_{\alpha,\beta}+1)
e_{\alpha+\beta}.\tag{a}
\end{equation*}
This can be checked by an explicit verification. Let us just consider
the case where $\alpha,\beta\in \Phi^+$ and neither $\alpha$ nor $\beta$
belongs to $\{\alpha_1,\alpha_2\}$. This only occurs if
$\{\alpha,\beta\}=\{\alpha_1+\alpha_2,\alpha_1+2\alpha_2\}$.
Now, using $e_{\alpha_1+\alpha_2}=c_1[e_{\alpha_1},c_{\alpha_2}]$ and the 
Jacobi identity, we obtain 
\begin{align*}
[e_{\alpha_1+\alpha_2},e_{\alpha_1+2\alpha_2}]& =c_1[[e_{\alpha_1},
e_{\alpha_2}],e_{\alpha_1+2\alpha_2}]=\ldots =3c_2e_{2\alpha_1+3\alpha_2}.
\end{align*}
On the other hand, we also find that 
\[ \vec{\rho}(\alpha_1+\alpha_2,\alpha_1+2\alpha_2)= \left\{
\begin{array}{cl} -3 & \quad \mbox{if $c_1=1$, $c_2=-1$},\\ -6 &\quad
\mbox{if $c_1=-1$, $c_2=1$}. \end{array}\right.\]
Hence, the desired identity holds for $\alpha=\alpha_1+\alpha_2$ and
$\beta=\alpha_1+2\alpha_2$. The argument is analogous for 
$\alpha=\alpha_1+2\alpha_2$ and $\beta=\alpha_1+\alpha_2$. Finally, 
if $\alpha\in \Phi^+$ or $\beta\in \Phi^+$ belongs to $\{\alpha_1, 
\alpha_2\}$, then the verification is much easier and will be omitted.
The cases where at least one of $\alpha,\beta$ is a negative root is
handled by Remark~\ref{allsigns}.
\end{exmp}

\begin{table}[htbp] \caption{Structure constants for $F_4$ (see 
Example~\ref{expF4})} \label{Mf4} 
\begin{center} {\footnotesize
$\begin{array}{ccrrr} \hline \alpha &\beta & \vec{\rho}_{\alpha\beta}
& \vec{\rho}_{\beta\alpha} &N_{\alpha,\beta}\\\hline
    1000&  0100&    -2&     0&    -1 \\
    1000&  0110&    -2&     0&    -1 \\
    1000&  0120&    -2&     0&    -1 \\
    1000&  0111&    -2&     0&    -1 \\
    1000&  0121&    -2&     0&    -1 \\
    1000&  0122&    -2&     0&    -1 \\
    1000&  1342&    -6&     0&    -1 \\
    0100&  0010&     0&    -2&     1 \\
    0100&  0011&     0&    -2&     1 \\
    0100&  1120&     0&    -6&     1 \\
    0100&  1121&     0&    -6&     1 \\
    0100&  1122&     0&    -6&     1 \\
    0100&  1242&     0&   -10&     1 \\
    0010&  0001&    -1&     0&    -1 \\
    0010&  1100&    -2&     0&    -1 \\
    0010&  0110&    -2&     0&    -2 \\
    0010&  1110&    -2&     0&    -2 \\
    0010&  0111&    -3&     0&    -1 \\
    0010&  1111&    -3&     0&    -1 \\
    0010&  1221&    -5&     0&    -1 \\
    0010&  1222&    -6&     0&    -1 \\
    0010&  1232&    -6&     0&    -2 \\
    0001&  0110&     0&    -1&     1 \\
    0001&  1110&     0&    -1&     1 \\
    0001&  0120&     0&    -2&     1 \\
    0001&  1120&     0&    -2&     1 \\
    0001&  0121&     0&    -2&     2 \\
    0001&  1220&     0&    -2&     1 \\
    0001&  1121&     0&    -2&     2 \\
    0001&  1221&     0&    -2&     2 \\
    0001&  1231&     0&    -3&     1 \\
    1100&  0011&     0&    -2&     1 \\
    1100&  0120&    -2&    -4&    -1 \\
    1100&  0121&    -2&    -4&    -1 \\ \hline
\end{array} \qquad \qquad \begin{array}{ccrrr} \hline \alpha &\beta & 
\vec{\rho}_{\alpha\beta} & \vec{\rho}_{\beta\alpha} &N_{\alpha,\beta}\\
\hline
    1100&  0122&    -2&    -4&    -1 \\
    1100&  1242&    -4&   -10&     1 \\
    0110&  0011&    -1&    -2&     1 \\
    0110&  1110&    -2&    -4&    -2 \\
    0110&  1111&    -3&    -4&    -1 \\
    0110&  1121&    -3&    -6&     1 \\
    0110&  1122&    -4&    -6&     1 \\
    0110&  1232&    -6&    -8&    -2 \\
    0011&  1110&    -2&    -1&    -1 \\
    0011&  0111&    -3&    -1&    -2 \\
    0011&  1111&    -3&    -1&    -2 \\
    0011&  1220&    -4&    -2&     1 \\
    0011&  1221&    -5&    -2&     1 \\
    0011&  1231&    -5&    -3&     2 \\
    1110&  0111&    -5&    -2&     1 \\
    1110&  0121&    -5&    -4&    -1 \\
    1110&  0122&    -6&    -4&    -1 \\
    1110&  1232&   -10&    -8&    -2 \\
    0120&  1111&    -6&    -4&    -1 \\
    0120&  1122&    -8&    -6&     1 \\
    0120&  1222&   -12&    -6&     1 \\
    0111&  1120&    -2&    -8&    -1 \\
    0111&  1111&    -3&    -5&    -2 \\
    0111&  1121&    -3&    -8&    -1 \\
    0111&  1231&    -5&   -11&     2 \\
    1120&  0122&   -10&    -4&    -1 \\
    1120&  1222&   -16&    -6&     1 \\
    1111&  0121&    -5&    -6&     1 \\
    1111&  1231&    -9&   -11&     2 \\
    0121&  1121&    -6&    -8&    -2 \\
    0121&  1221&   -10&    -8&    -2 \\
    1220&  0122&   -10&    -8&    -1 \\
    1220&  1122&   -10&   -12&    -1 \\
    1121&  1221&   -14&    -8&    -2 \\ \hline
\end{array}$}
\end{center}
\text{\scriptsize (Here, for example, $0121$ stands for the root 
$\alpha_2{+}2\alpha_3 {+}\alpha_4$.)}
\end{table}

\begin{exmp} \label{expF4} Let $\Phi$ be of type $F_4$. Let $I=\{1,2,3,4\}$ 
where the notation is such that $\alpha_1$, $\alpha_2$ are long, $\alpha_3,
\alpha_4$ are short and $\alpha_2,\alpha_3$ are connected in the graph 
of $A$. We assume that $c_1=c_3=-1$ and $c_2=c_4=1$, indicated in 
the graph of $A$ as follows:
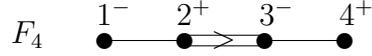
\begin{center}
\begin{picture}(150,30) 
\put(  5, 9){$F_4$}
\put( 40,10){\circle*{6}}
\put( 37,16){$1^-$}
\put( 41,10){\line(1,0){28}}
\put( 70,10){\circle*{6}}
\put( 67,16){$2^+$}
\put( 68,12){\line(1,0){32}}
\put( 68, 8){\line(1,0){32}}
\put( 80,7.3){$>$}
\put(100,10){\circle*{6}}
\put(98,16){$3^-$}
\put(100,10){\line(1,0){30}}
\put(130,10){\circle*{6}}
\put(128,16){$4^+$}
\end{picture}
\end{center}
Then $\vQ(A)=\{(1,2),(3,2),(3,4)\}$. Hence, if $\alpha=\sum_{1\leq i
\leq 4}n_i\alpha_i\in Q$ and $\beta=\sum_{1\leq j \leq 4} m_j\alpha_j\in Q$,
then
\[\vec{\rho}(\alpha,\beta)=-2n_1m_2-2n_3m_2-n_3m_4.\]
(Note that $d_1=d_2=2$, $d_3=d_4=1$, $a_{32}=-2$ and $a_{34}=-1$.) In 
Table~\ref{Mf4} we list the values $\vec{\rho}_{\alpha\beta}=\vec{\rho}
(\alpha,\beta)$ for $\alpha,\beta\in \Phi^+$ such that $\alpha+\beta
\in \Phi$. The last column contains the actual structure constants
$N_{\alpha,\beta}=\varepsilon(\alpha,\beta)(q_{\alpha,\beta}+1)$ of~$\fg$.
With some effort, these could be computed ``by hand''; or one can use the  
computer programs in \cite{mygreen}, \cite{chevlie}. Note that, in the 
framework of \cite[\S 4]{mygreen} or \cite[\S 2]{chevlie}, there is an 
algorithm~---~independent of the considerations in this paper~---~for 
computing the structure constants of any~$\fg$ with respect to Lusztig's 
canonical basis and, hence, also with respect to~$\cL$ (via ($*$) in the
proof of Theorem~\ref{thmcan}). We can actually extract an abstract 
formula from Table~\ref{Mf4}. For any integer $n \in \Z$ we define 
\[ \delta_4(n):=\left\{\begin{array}{ll} (-1)^{n/2} & \mbox{ if $n$ is
even},\\ (-1)^{(n+1)/2} & \mbox{ if $n$ is odd}. \end{array}\right.\]
In particular, if $n$ is even, then $\delta_4(n)=(-1)^n$. Then one can check 
(``by hand'' using Table~\ref{Mf4}, or with the programs in \cite{mygreen},
\cite{chevlie}) that the following formula holds for all $\alpha,\beta 
\in \Phi$ such that $\alpha+\beta \in \Phi$:
\begin{equation*}
\varepsilon(\alpha,\beta)=\left\{\begin{array}{rl} \delta_4(\vec{\rho}
(\alpha,\beta)) & \quad\mbox{if $\vec{\rho}(\alpha,\beta)$ is even},\\ 
-\delta_4(\vec{\rho}(\beta,\alpha)) & \quad \mbox{if $\vec{\rho}(\alpha,
\beta)$ is odd};\end{array}\right.\tag{a}
\end{equation*}
furthermore, this formula remains valid if we replace each $c_i$ by $-c_i$.
\end{exmp}

In the above two examples there is a definite relation between the
signs $\varepsilon(\alpha,\beta)$ and the function $\vec{\rho}\colon Q
\times Q\rightarrow \Z$. However, by inspection of Table~\ref{Mf4} we 
see that the signs $\varepsilon(\alpha,\beta)$ are not determined by just 
looking at $\vec{\rho}(\alpha,\beta)$ alone~---~contrary to the situation 
in the simply laced case (see Corollary~\ref{vecA1}). Instead, exactly as
in Ringel \cite[p.~139]{Ri}, it seems to be necessary to take into account 
both $\vec{\rho}(\alpha,\beta)$ and $\vec{\rho}(\beta,\alpha)$, and also 
not just their values modulo~$2$. Inspired by the formulae of Rylands 
\cite[\S 3]{Ry}, and using extensive experiments with the programs in 
\cite{mygreen}, \cite{chevlie}, we are lead to the formulation of the 
following result.

\begin{prop} \label{prop41} Let $\fg$ be of type $B_r$ or~$C_r$, where 
$r\geq 2$. Let $\cL=\{e_\alpha\mid\alpha\in \Phi\}$ be a special Chevalley 
system. Then the signs $\varepsilon(\alpha,\beta)$ in the corresponding
structure constants are given as follows, where $\alpha,\beta\in \Phi^+$ 
are such that $\alpha+\beta\in\Phi$. 
\begin{itemize}
\item[(a)] If $\fg$ is of type $B_r$, then $\vec{\rho}(\alpha,\beta)\in 
2\Z$ and $\varepsilon(\alpha,\beta)=(-1)^{\vec{\rho}(\alpha,\beta)/2}$.
\item[(b)] If $\fg$ is of type $C_r$, then $\vec{\rho}(\alpha,\beta)-
\vec{\rho}(\beta,\alpha)\in \{\pm 1,\pm 2,\pm 3\}$ and 
\[ \varepsilon(\alpha,\beta)=\left\{\begin{array}{rl} 1 & \quad
\mbox{if $\vec{\rho}(\alpha,\beta)>\vec{\rho}(\beta,\alpha)$},\\ 
-1 & \quad \mbox{if $\vec{\rho}(\alpha,\beta)<\vec{\rho}(\beta,\alpha)$}.
\end{array}\right.\]
\end{itemize}
\end{prop}

Using Remark~\ref{remcan1}(a) and Remark~\ref{lemRi}, one immediately sees 
that it is sufficient to prove Proposition~\ref{prop41} for one of the two 
possible choices of signs $\{c_i\mid i \in I\}$ fixing the special Chevalley 
system~$\cL$. The further strategy is analogous to that in Rylands \cite{Ry}. 
The crucial step is to work out the values $\vec{\rho}(\alpha,\beta)$ 
(in~(a)) or, at least, the differences $\vec{\rho}(\alpha,\beta)-\vec{\rho}
(\beta,\alpha)$ (in~(b)), using the explicit knowledge of all roots in 
type $B_r$ and $C_r$ as expressions of simple roots. It then 
remains to apply the appropriate ``folding'' and Theorem~\ref{thmF1};  
see Section~\ref{secBC} for details.

\begin{rem} \label{prop41a} Let $\fg$ be of type $B_r$ or $C_r$, as above.
Let $\alpha,\beta\in \Phi$ be arbitrary such that $\alpha+\beta\in\Phi$. We 
will see in the following section that the statement in 
Proposition~\ref{prop41}(a) remains valid without modification. In 
Proposition~\ref{prop41}(b), some modifications are required. For any
$\alpha,\beta \in Q=\Z\Phi$ let us set 
\[\mu(\alpha,\beta):=\left\{\begin{array}{rl} 1 & \quad \mbox{if 
$\vec{\rho}(\alpha,\beta)\geq \vec{\rho}(\beta,\alpha)$},\\ -1 & \quad 
\mbox{if $\vec{\rho}(\alpha,\beta)<\vec{\rho}(\beta,\alpha)$}.\end{array}
\right.\]
Using the formulae in Remark~\ref{allsigns}, and the fact that 
$\vec{\rho}\colon Q \times Q\rightarrow \Z$ is bilinear, one easily deduces
the following formulae: 
\[\varepsilon(\alpha,\beta)=\left\{\begin{array}{rl} 
\mu(\alpha,\beta) & \mbox{if $\mbox{sgn}(\alpha)=\mbox{sgn}(\beta)$},\\ 
(-1)^{\hgt(\alpha)} \mu(\alpha,\beta) & \mbox{if $\mbox{sgn}(\alpha)
\neq \mbox{sgn}(\beta)= \mbox{sgn}(\alpha+\beta)$},\\
(-1)^{\hgt(\beta)} \mu(\alpha,\beta) & \mbox{if $\mbox{sgn}(\beta)
\neq \mbox{sgn}(\alpha)=\mbox{sgn}(\alpha+\beta)$},\end{array}\right.\]
where $\mbox{sgn}\colon \Phi\rightarrow \{\pm 1\}$ is defined in
Remark~\ref{rem2}.
\end{rem}

\begin{rem} \label{summa} We can summarize the above results as follows.
Let $\fg$ be \textit{any} simple Lie algebra and $\cL=\{e_\alpha\mid 
\alpha\in \Phi\}$ be a special Chevalley system for~$\fg$. Then there is a 
function $\mathfrak{f}\colon\Z\times\Z \rightarrow \{\pm 1\}$ such that 
\[\varepsilon(\alpha,\beta)=\mathfrak{f}\bigl(\vec{\rho}(\alpha,\beta),
\vec{\rho}(\beta,\alpha)\bigr)\quad\mbox{whenever}\quad \alpha,\beta,
\alpha+\beta\in \Phi^+.\]
This function $\mathfrak{f}$ only depends on the type of $\fg$, that is, 
the letter $A$, $B$, $C$, $D$, $E$, $F$ or $G$, but not on the rank of~$\fg$.
For example, for type $A$, $D$, $E$ or $G$, we have $\mathfrak{f}(n,m)=
(-1)^{n}$; see Corollary~\ref{vecA1} and Example~\ref{expG2}(a). Or for 
type $C$, we have $\mathfrak{f}(n,m)=1$ if $n>m$, and 
$\mathfrak{f}(n,m)=-1$ if $n<m$; see Proposition~\ref{prop41}(b). 
\end{rem}

\section{Computing structure constants for $B_r$ and $C_r$} \label{secBC}

In this section, we provide the details of the proof of 
Proposition~\ref{prop41}. Let $\fg$ be a simple Lie algebra with a 
root system $\Phi$ of type $B_r$ or $C_r$, where $r\geq 2$. As explained
in the previous section, a special Chevalley system $\cL=\{e_\alpha\mid 
\alpha \in \Phi\}$ for $\fg$ can be obtained by a folding procedure from 
a suitable Lie algebra $\fg^\circ$ with a simply laced root system 
$\Phi^\circ$; see Theorem~\ref{thmF1}. If $\Phi$ is of type $B_r$, then
$\Phi^\circ$ will be of type $D_{r+1}$; if $\Phi$ is of type~$C_r$, then 
$\Phi^\circ$ will be of type $A_{2r-1}$. We fix labellings of the 
simple roots, and the signs $\{c_i\}$ (in terms of plus/minus signs 
attached to the vertices in the graphs) as in Table~\ref{Mfoldi}. 

\begin{table}[htbp]
\caption{Folding for type $B_r$ and $C_r$} \label{Mfoldi}
\begin{center} 
\begin{picture}(340,90) 
\put( 13,43){$C_r$}
\put( 12,32){$\scriptstyle{r \geq 2}$}
\put( 40,40){\circle*{5}}
\put( 38,45){$\scriptstyle{1}^+$}
\put( 41,40){\line(1,0){18}}
\put( 60,40){\circle*{5}}
\put( 58,45){$\scriptstyle{2}^-$}
\put( 60,40){\line(1,0){10}}
\put( 80,40){\circle*{1}}
\put( 90,40){\circle*{1}}
\put(100,40){\circle*{1}}
\put(110,40){\line(1,0){10}}
\put(120,40){\circle*{5}}
\put(108,45){$\scriptstyle{r{-}1}^\pm$}
\put(120,38){\line(1,0){20}}
\put(120,41){\line(1,0){20}}
\put(126,36.5){$<$}
\put(140,40){\circle*{5}}
\put(138,45){$\scriptstyle{r}^\mp$}

\put(181,36){$A_{2r{-}1}$}
\put(227,40){\circle*{1}}
\put(227,38){\circle*{1}}
\put(227,36){\circle*{1}}
\put(247,40){\circle*{1}}
\put(247,38){\circle*{1}}
\put(247,36){\circle*{1}}
\put(307,40){\circle*{1}}
\put(307,38){\circle*{1}}
\put(307,36){\circle*{1}}

\put(227,47){\circle*{5}}
\put(225,52){$\scriptstyle{1}^+$}
\put(227,47){\line(1,0){19}}
\put(247,47){\circle*{5}}
\put(245,52){$\scriptstyle{2}^-$}
\put(247,47){\line(1,0){10}}
\put(267,47){\circle*{1}}
\put(277,47){\circle*{1}}
\put(287,47){\circle*{1}}
\put(297,47){\line(1,0){10}}
\put(307,47){\circle*{5}}
\put(298,52){$\scriptstyle{r{-}1}^\mp$}
\put(308,47){\line(2,-1){19}}
\put(327,38){\circle*{5}}
\put(327,42){$\scriptstyle{r}^\pm$}

\put(227,29){\circle*{5}}
\put(210,17){$\scriptstyle{2r{-}1}^+$}
\put(227,29){\line(1,0){19}}
\put(247,29){\circle*{5}}
\put(238,17){$\scriptstyle{2r{-}2}^-$}
\put(247,29){\line(1,0){10}}
\put(267,29){\circle*{1}}
\put(277,29){\circle*{1}}
\put(287,29){\circle*{1}}
\put(297,29){\line(1,0){10}}
\put(307,29){\circle*{5}}
\put(298,17){$\scriptstyle{r{+}1}^\mp$}
\put(308,29){\line(2,1){19}}

\put( 13,83){$B_{r}$}
\put( 12,72){$\scriptstyle{r \geq 2}$}
\put( 40,80){\circle*{5}}
\put( 38,85){$\scriptstyle{1}^+$}
\put( 41,81){\line(1,0){18}}
\put( 41,78){\line(1,0){18}}
\put( 46,76.5){$<$}
\put( 60,80){\circle*{5}}
\put( 58,85){$\scriptstyle{2}^-$}
\put( 61,80){\line(1,0){17}}
\put( 80,80){\circle*{5}}
\put( 77,85){$\scriptstyle{3}^+$}
\put( 81,80){\line(1,0){11}}
\put(100,80){\circle*{1}}
\put(110,80){\circle*{1}}
\put(120,80){\circle*{1}}
\put(128,80){\line(1,0){10}}
\put(140,80){\circle*{5}}
\put(138,85){$\scriptstyle{r}^\pm$}

\put(180,81){$D_{r{+}1}$}
\put(227,84){\circle*{1}}
\put(227,82){\circle*{1}}
\put(227,80){\circle*{1}}
\put(227,92){\circle*{5}}
\put(213,90){$\scriptstyle{1}^+$}
\put(227,72){\circle*{5}}
\put(213,70){$\scriptstyle{0}^+$}
\put(229,91){\line(2,-1){18}}
\put(229,73){\line(2,1){18}}
\put(247,82){\circle*{5}}
\put(245,87){$\scriptstyle{2}^-$}
\put(247,82){\line(1,0){30}}
\put(267,82){\circle*{5}}
\put(265,87){$\scriptstyle{3}^+$}
\put(287,82){\circle*{1}}
\put(297,82){\circle*{1}}
\put(307,82){\circle*{1}}
\put(315,82){\line(1,0){10}}
\put(327,82){\circle*{5}}
\put(325,87){$\scriptstyle{r}^\pm$}
\put( 25, 0){\footnotesize (The dotted vertical lines are between vertices 
in the same $\tau$-orbit.)}
\end{picture}
\end{center}
\end{table}

In the discussion below, the paragraphs referring to $\Phi$ of
type $B_r$ will be labelled by \textbf{B.1}, \textbf{B.2} and so
on; similarly, the paragraphs referring to $\Phi$ of type $C_r$ will 
be labelled by \textbf{C.1}, \textbf{C.2} and so on. 

\renewcommand{\thesection}{B}

\begin{abs} \label{absB1} Assume that $\Phi$ is of type $B_r$ ($r\geq 2$),
with simple roots $I=\{\alpha_1,\ldots,\alpha_r\}$ labelled as in 
Table~\ref{Mfoldi}. We have $d_1=1$ and $d_i=2$ for $2\leq i \leq r$ 
in this case; 
furthermore,
\[\vQ(A) =\{(2,1),(2,3),(4,3),(4,5),(6,5),(6,7), \ldots\}.\]
Since $a_{12}=-2$, $a_{21}=-1$ and $a_{i,i+1}=a_{i+1,i}=-1$ for $2 \leq i 
\leq r-1$, we have $d_ia_{ij}=-2$ for all $(i,j) \in \vQ(A)$ and so 
\[ \vec{\rho}(\alpha,\beta)=-2\sum_{(i,j)\in \vQ(A)} n_im_j 
\quad \in 2\Z \qquad \mbox{for all $\alpha,\beta \in \Z\Phi$}\]
where $\alpha=\sum_{1\leq i \leq r} n_i\alpha_i$ and $\beta=\sum_{1\leq i
\leq r} m_i\alpha$ with $n_i,m_i\in\Z$. The positive roots in $\Phi$ are 
described as follows:
\begin{align*}
&\alpha_1+\alpha_2+\ldots +\alpha_i\quad \mbox{for $1\leq i \leq r$},\\ 
&\alpha_{i+1}+\alpha_{i+2}+\ldots + \alpha_j \quad\mbox{for $1 \leq 
i<j\leq r$},\\&2(\alpha_1+\ldots +\alpha_i)+\alpha_{i+1} + \ldots + 
\alpha_j  \quad \mbox{for $1\leq i<j \leq r$};
\end{align*} 
see, e.g., ``case (b)'' in \cite[Remark~2.5.5]{mynotes}. 
\end{abs}

\begin{abs} \label{absB2} Let $\Phi^\circ$ be a root system of type
$D_{r+1}$ ($r\geq 2$), with simple roots $I^\circ=\{\alpha_0^\circ,
\alpha_1^\circ,\ldots,\alpha_r^\circ\}$ labelled as in Table~\ref{Mfoldi};
let $A^\circ$ be the corresponding Cartan matrix. (We typically attach a 
superscript ${}^\circ$ to all objects related to $\Phi^\circ$.) We have 
\[\vQ(A^\circ)=\{(2,0),(2,1),(2,3),(4,3),(4,5),(6,5),(6,7),\ldots\}.\]
The positive roots in $\Phi^\circ$ are given by
\begin{align*}
&\alpha_{i+1}^\circ+\alpha_{i+2}^\circ+\ldots +\alpha_j^\circ,\\
&(\alpha_0^\circ+\alpha_1^\circ+\ldots +\alpha_i^\circ)+(\alpha_2^\circ+
\alpha_3^\circ+\ldots+ \alpha_j^\circ),
\end{align*}
where $0\leq i <j \leq r$ in both cases; see, e.g., ``case (a)'' 
in \cite[Remark~2.5.5]{mynotes}. The permutation $\tau\colon \Phi^\circ
\rightarrow \Phi^\circ$ switches $\alpha_0^\circ,\alpha_1^\circ$ and 
fixes~$\alpha_i^\circ$ for $2\leq i \leq r$. The linear map $\nu \colon 
\Phi^\circ \rightarrow \Phi$ in (F1) is defined by 
\[\nu(\alpha_0^\circ)=\nu(\alpha_1^\circ)=\alpha_1\qquad \mbox{and}\qquad
\nu(\alpha_i^\circ)=\alpha_i \quad \mbox{for $2\leq i \leq r$}.\]
The above description of the roots in $\Phi$ and $\Phi^\circ$ shows the 
following. Let $\alpha=\sum_{1\leq i \leq r} n_i\alpha_i\in \Phi$, 
where either $n_i\geq 0$ for all~$i$, or $n_i\leq 0$ for all~$i$. Then 
$\alpha=\nu(\alpha^\circ)$ where
\[ \alpha^\circ:=n_0^\circ\alpha_0^\circ+n_1^\circ\alpha_1^\circ
+\sum_{2\leq i \leq r} n_i\alpha_i^\circ\quad\in \Phi^\circ\]
and $n_0^\circ,n_1^\circ\in \Z$ are such that $n_1=n_0^\circ+n_1^\circ$; 
furthermore, $|n_1|\leq 2$ and $|n_0^\circ|,|n_1^\circ|\leq 1$. Thus,
there are exactly two possibilities for $\alpha^\circ$ if $|n_1|=1$;
otherwise, there is a unique possibility.
\end{abs}

\begin{abs} \label{absB3} We can now complete the proof of 
Proposition~\ref{prop41}(a), even for arbitrary $\alpha,\beta\in \Phi$ 
such that $\alpha+\beta\in \Phi$. We write $\alpha=\sum_{1\leq i \leq r} 
n_i\alpha_i$ and $\beta=\sum_{1\leq i \leq r} m_i\alpha_i$ with $n_i,m_i
\in \Z$. As above we have $\alpha=\nu(\alpha^\circ)$ and $\beta=
\nu(\beta^\circ)$ where 
\begin{alignat*}{2} 
\alpha^\circ&=n_0^\circ\alpha_0^\circ+n_1^\circ\alpha_1^\circ +\sum_{2\leq 
i \leq r} n_i\alpha_i^\circ\in \Phi^\circ &&\qquad (n_1=n_0^\circ+n_1^\circ),
\\ \beta^\circ&=m_0^\circ\alpha_0^\circ+m_1^\circ\alpha_1^\circ +
\sum_{2\leq i \leq r} m_i\alpha_i^\circ\in \Phi^\circ &&\qquad (m_1=
m_0^\circ+m_1^\circ);
\end{alignat*}
note that we can always choose $n_0^\circ,n_1^\circ,m_0^\circ,m_1^\circ$ 
such that $\alpha^\circ+\beta^\circ\in \Phi^\circ$, in accordance with
(F4). Now, we have 
\[ \vQ(A)=\{(2,1)\}\cup P\qquad\mbox{and}\qquad \vQ(A^\circ)=
\{(2,0),(2,1)\} \cup P\]
where $P:=\{(2,3),(4,3),(4,5),(6,5),(6,7),\ldots\}$. Hence, we obtain
\begin{align*}
\vec{\rho}(\alpha,\beta)&=-2n_2m_1-2\sum_{(i,j)\in P} n_im_j,\\
\vec{\rho^\circ}(\alpha^\circ,\beta^\circ)&=-n_2m_0^\circ-n_2m_1^\circ
-\sum_{(i,j)\in P} n_im_j.
\end{align*}
This yields that 
$\vec{\rho^\circ}(\alpha^\circ,\beta^\circ)=\textstyle{\frac{1}{2}}
\vec{\rho}(\alpha,\beta)+n_2(m_1-m_0^\circ-m_1^\circ)=
\textstyle{\frac{1}{2}} \vec{\rho}(\alpha,\beta)$.
So Theorem~\ref{thmF1} and Corollary~\ref{vecA1} show that 
$\varepsilon(\alpha,\beta)=(-1)^{\vec{\rho}(\alpha,\beta)/2}$. \qed
\end{abs}

\renewcommand{\thesection}{C}
\setcounter{thm}{0}

\begin{abs} \label{absC1} Assume that $\Phi$ is of type $C_r$ ($r\geq 2$),
with simple roots $I=\{\alpha_1,\ldots,\alpha_r\}$ labelled as in 
Table~\ref{Mfoldi}. We have $d_r=2$ and $d_i=1$ for $1\leq i \leq r-1$ in 
this case; furthermore,
\[\vQ(A) =\{(2,1),(2,3),(4,3),(4,5),(6,5),(6,7), \ldots\}.\]
Since $a_{r,r-1}=-1$, $a_{r-1,r}=-2$ and $a_{i,i+1}=a_{i+1,i}=-1$ for 
$1 \leq i \leq r-2$, we have $d_ra_{r,r-1}=d_{r-1}a_{r-1,r}=-2$ and 
$d_ia_{ij}=-1$ for $(i,j)\in \vQ(A)$ with $i\leq r-2$. This yields 
the formula
\[ \vec{\rho}(\alpha,\beta)=\left\{\begin{array}{cl} -2n_rm_{r-1}-
\sum_{(i,j)\in \vQ(A),i\leq r-2} n_im_j & \quad \mbox{if $r$ is even},
\\ -2n_{r-1}m_r-\sum_{(i,j)\in \vQ(A),i\leq r-2} n_im_j & \quad 
\mbox{if $r$ is odd},\end{array}\right.\]
where $\alpha=\sum_{1\leq i \leq r} n_i\alpha_i$ and $\beta=\sum_{1\leq i
\leq r} m_i\alpha$ with $n_i,m_i\in\Z$. (Note that the sign $c_r$ equals 
$-1$ if $r$ is even, and $1$ if $r$ is odd.) The positive roots in $\Phi$ 
are described as follows (see, e.g., ``case (c)'' in 
\cite[Remark~2.5.5]{mynotes}). The long roots are given by 
\[ \gamma_j:=2(\alpha_j+\alpha_{j+1}+\ldots +\alpha_{r-1})+\alpha_r
\qquad \mbox{for $1\leq j \leq r$}.\]
(For example, $\gamma_r=\alpha_r$.) There are two types of short roots:
\begin{alignat*}{2}
\alpha_{ij}&:=\alpha_{i}+\alpha_{i+1}+\ldots + \alpha_{j-1} &&\qquad 
\mbox{for $1\leq i<j\leq r$},\tag{I}\\\gamma_{ij}&:=\alpha_i+\alpha_{i+1}
+\ldots+\alpha_{j-1}+\gamma_j &&\qquad\mbox{for $1\leq i<j\leq r$}.\tag{II}
\end{alignat*}
We may also set $\alpha_{jj}:=0$ and $\gamma_{jj} :=\gamma_j$ for 
$1\leq j \leq r$ in order to have a uniform notation. Then 
$\gamma_{ij}=\alpha_{ij}+\gamma_j$ for $1\leq i\leq j \leq r$.  Hence, 
\[ \Phi^+=\{\alpha_{ij}\mid 1\leq i<j\leq r \}\cup \{\gamma_{ij} \mid
1\leq i\leq j \leq r\}.\]
\end{abs}

\begin{abs} \label{absC3} Let $\alpha,\beta\in \Phi^+$ be such that 
$\alpha+\beta\in \Phi$. The above description shows that $\alpha_r$ occurs
with multiplicity $0$ or $1$ in every positive root. So at least one of 
$\alpha,\beta$ must be of type (I). Let us assume that $\alpha=\alpha_{ij}$
where $1\leq i <j \leq r$. Then we have the following possibilities 
for~$\beta$.
\begin{itemize}
\item[(1)] $\beta=\alpha_{kl}$ where $1\leq k<l\leq r$ and either $k=j$ or
$l=i$; this is analogous to the situation in Example~\ref{expAn}. So
$\vec{\rho}(\alpha,\beta)$ and $\vec{\rho}(\beta,\alpha)$ are determined
by Example~\ref{expAn}(a). This yields:
\[ \vec{\rho}(\alpha,\beta)-\vec{\rho}(\beta,\alpha)=\left\{
\begin{array}{rl} (-1)^j &\quad \mbox{if $k=j$},\\ -(-1)^l & 
\quad \mbox{if $l=i$}. \end{array}\right.\]
\item[(2)] $\beta=\gamma_{jk}=\alpha_{jk}+\gamma_k$ where $j\leq k\leq r$; 
then 
\[\alpha+\beta = (\alpha_{ij}+\alpha_{jk})+\gamma_k=\alpha_{ik}+\gamma_k=
\gamma_{ik}.\]
As far as the evaluation of $\vec{\rho}(\alpha,\beta)$ and $\vec{\rho}
(\beta, \alpha)$ is concerned, the situation is almost like that in 
Example~\ref{expAn}. (There is still at most one non-zero term in the 
sum defining $\vec{\rho}(\alpha,\beta)$ or $\vec{\rho}(\beta,\alpha)$, 
but this term may be~$-2$.) One finds that 
\[ \vec{\rho}(\alpha,\beta)-\vec{\rho}(\beta,\alpha)=
\left\{\begin{array}{rl} (-1)^j &\quad \mbox{if $j<k$},\\ (-2)^j & 
\quad \mbox{if $j=k$}. \end{array}\right.\]
\item[(3)] $\beta=\gamma_{kj}=\alpha_{kj}+\gamma_j$ where $1\leq k\leq j$.
If $i \leq k$, then 
\begin{align*} 
\alpha+\beta &=(\alpha_i+\ldots +\alpha_{j-1})+(\alpha_k+\ldots + 
\alpha_{j-1})+\gamma_j\\&=(\alpha_i+\ldots +\alpha_{k-1})+
2(\alpha_k+\ldots +\alpha_{j-1})+\gamma_j\\&=(\alpha_i+\ldots+
\alpha_{k-1})+\gamma_k=\alpha_{ik}+\gamma_k.
\end{align*}
Similarly, if $i>k$, then $\alpha+\beta=\alpha_{ki}+\gamma_i$. 
Now it is somewhat more complicated to describe the exact values of 
$\vec{\rho}(\alpha,\beta)$ and $\vec{\rho}(\beta,\alpha)$. (There may
be several non-zero terms contributing to the sum defining 
$\vec{\rho}(\alpha,\beta)$.) Nevertheless, a straightforward but 
slightly lengthy verification (which we omit) shows that 
\[ \qquad\vec{\rho}(\alpha,\beta)-\vec{\rho}(\beta,\alpha)\;= \; 
(-1)^j \quad \mbox{or} \quad 2(-1)^j \quad\mbox{or}\quad 3(-1)^j.\]
\end{itemize} 
\end{abs}

\begin{abs} \label{absC5} Here are a few examples for type $C_4$:
\[ \begin{array}{ccrrr} \hline \alpha=\alpha_{ij} & \beta & 
\vec{\rho}(\alpha,\beta) & \vec{\rho}(\beta,\alpha) & N_{\alpha,\beta}
\\ \hline \alpha_1 & \alpha_2 & 0 \quad & -1\quad & 1\;\\ 
\alpha_2 & \alpha_2{+}2\alpha_3{+}\alpha_4  & -2\quad & 0\quad & -2 \;\\
\alpha_2 & \alpha_1{+}\alpha_2{+}2\alpha_3{+}\alpha_4 & -3 \quad& 0\quad 
& -1\;\\ \alpha_1{+}\alpha_2 & \alpha_2{+}2\alpha_3{+}\alpha_4 & -2 
\quad & -1 \quad & -1\; \\ \alpha_1{+}\alpha_2{+}\alpha_3 & \alpha_2{+}
\alpha_3{+}\alpha_4 & -1\quad & -4 \quad & 1 \;\\ \alpha_1{+}\alpha_2
{+}\alpha_3 & \alpha_1{+}\alpha_2{+} \alpha_3{+} \alpha_4 & -2 \quad & -4 
\quad & 2 \;\\ \hline\end{array}\]
In particular, these show that the sign $\varepsilon(\alpha,\beta)$ of 
$N_{\alpha,\beta}$ is not determined by $\vec{\rho}(\alpha,\beta)$ or 
$\vec{\rho}(\beta,\alpha)$ alone. They also show that the difference 
$\vec{\rho}(\alpha,\beta)- \vec{\rho}(\beta,\alpha)$ can take any of 
the values $\pm 1, \pm 2,\pm 3$.
\end{abs}

\begin{abs} \label{absC2} Let $\Phi^\circ$ be a root system of type
$A_{2r-1}$ ($r\geq 2$), with simple roots $I^\circ=\{\alpha_1^\circ,
\ldots,\alpha_{2r-1}^\circ\}$ labelled as in Table~\ref{Mfoldi}; let 
$A^\circ$ be the corresponding Cartan matrix. (As before, we typically 
attach a superscript ${}^\circ$ to all objects related to $\Phi^\circ$.) 
The set $\vQ(A^\circ)$ and the positive roots in $\Phi^\circ$ are
described in Example~\ref{expAn}; the latter are just given by
\[ \alpha_{ij}^\circ:=\alpha_{i}^\circ+\alpha_{i+1}^\circ+\ldots +
\alpha_{j-1}^\circ \qquad \mbox{for $1\leq i <j \leq 2r$}.\] 
The permutation $\tau\colon \Phi^\circ\rightarrow \Phi^\circ$ fixes
$\alpha_r^\circ$ and switches $\alpha_i^\circ,\alpha_{2r-i}^\circ$ for 
$1\leq i \leq r-1$. The linear map $\nu \colon \Phi^\circ\rightarrow \Phi$ 
in (F1) is defined by 
\[ \nu(\alpha_r^\circ)=\alpha_r \quad \mbox{and} \quad \nu(\alpha_i^\circ)
=\nu(\alpha_{2r-i}^\circ)=\alpha_i\quad \mbox{for $1\leq i \leq r-1$}.\]
Thus, the description of the roots in $\Phi^\circ$ is simpler than in 
the previous case, but $\nu\colon \Phi^\circ \rightarrow \Phi$ is more 
complicated. Note the following formulae which will be used frequently
below. We have 
\[ \alpha_{ij}=\nu(\alpha_{ij}^\circ)=\nu(\alpha_{2r-j+1,2r-i+1}^\circ)
\qquad \mbox{for $1\leq i<j\leq r$}.\]
Furthermore, for $1\leq j \leq r$, we have $\nu(\gamma_j^\circ)=\gamma_j$ 
where we set 
\begin{align*}
\gamma_j^\circ &:=(\underbrace{\alpha_j^\circ+\alpha_{j+1}^\circ+
\ldots + \alpha_{r-1}^\circ}_{\text{$r-j$ terms}}) +\alpha_r^\circ+
(\underbrace{\alpha_{r+1}^\circ+ \alpha_{r+2}^\circ+\ldots +
\alpha_{2r-j}^\circ}_{\text{$r-j$ terms}})\\ &=\alpha_{j,2r-j+1}^\circ
\in \Phi^\circ.
\end{align*}
\end{abs}

\begin{abs} \label{absC4} We can now complete the proof of 
Proposition~\ref{prop41}(b). Let $\alpha,\beta\in \Phi^+$ be such that 
$\alpha+\beta\in \Phi$. The formulation of the desired formula in 
Proposition~\ref{prop41}(b) shows that it is enough to prove that formula 
for either $\varepsilon(\alpha,\beta)$ or $\varepsilon(\beta,\alpha)$. In
\textbf{\ref{absC3}} we have seen that one of $\alpha,\beta$ must be equal 
to~$\alpha_{ij}$ where $1\leq i<j\leq r$. So we can assume that 
$\alpha=\alpha_{ij}$; then~$\beta$ will be given as in one of the above 
three cases (1), (2), (3). We have $\nu(\alpha^\circ)=\alpha$ for 
$\alpha^\circ:=\alpha_{ij}^\circ\in \Phi^\circ$. For each possibility
of $\beta$ we need to find $\beta^\circ \in \Phi^\circ$ such that
$\nu(\beta^\circ)=\beta$ and $\alpha^\circ+\beta^\circ\in\Phi^\circ$.

First let $\beta=\alpha_{kl}$ as in case~(1), where $1\leq k<l\leq r$ and 
either $k=j$ or $l=i$. Then we can just take $\beta^\circ=\alpha_{kl}^\circ
\in \Phi^\circ$. Consequently, we have $\vec{\rho}(\alpha,\beta)=
\vec{\rho^\circ}(\alpha^\circ,\beta^\circ)$ and $\vec{\rho}(\beta,
\alpha)=\vec{\rho^\circ}(\beta^\circ,\alpha^\circ)$. In particular:
\[ \vec{\rho}(\alpha,\beta)-\vec{\rho}(\alpha,\beta)=\vec{\rho^\circ}
(\alpha^\circ,\beta^\circ)-\vec{\rho^\circ}(\beta^\circ,\alpha^\circ)\]
where the right hand side equals $\pm 1$ by Example~\ref{expAn}(b).
Using Theorem~\ref{thmF1} and the formula in Example~\ref{expAn}(c), it 
follows that the desired formula in Proposition~\ref{prop41}(b) holds.

Now let $\beta=\gamma_{jk}$ as in case~(2), or $\beta=\gamma_{kj}$ as 
in case~(3). In both cases, we set $\beta^\circ:=\alpha_{j,2r-k+1}^\circ
\in\Phi^\circ$. In case (2), we have 
\[\beta^\circ=\alpha_{j}^\circ+\alpha_{j+1}^\circ+\ldots+
\alpha_{k-1}^\circ+\gamma_k^\circ\qquad\mbox{where}\qquad j\leq k\leq r.\]
Hence, we see that $\nu(\beta^\circ)=\alpha_{jk}+\gamma_k=\gamma_{jk}=
\beta$ and 
\[ \alpha^\circ+\beta^\circ=(\alpha_{ij}^\circ+\alpha_{jk}^\circ)+
\gamma_k^\circ=\alpha_{ik}^\circ+\gamma_k^\circ=\alpha_{i,2r-k+1}^\circ
\in\Phi^\circ.\]
In case (3) we have 
\[ \beta^\circ=\gamma_j^\circ+\alpha_{2r-j+1}^\circ +\ldots +
\alpha_{2r-k}^\circ\qquad \mbox{where} \qquad 1\leq k\leq j.\]
Since $\nu(\alpha_{2r-l}^\circ)=\nu(\alpha_l^\circ)=\alpha_l$ for 
$1\leq l \leq r$, we have again $\nu(\beta^\circ)=\gamma_j+\alpha_{kj}=
\beta$ and $\alpha^\circ+\beta^\circ=\alpha_{i,2r-k+1}^\circ\in\Phi^\circ$.

Now, by Example~\ref{expAn}(a) we have
\begin{align*}
\vec{\rho^\circ}(\alpha^\circ,\beta^\circ)&= \vec{\rho^\circ}
(\alpha_{ij}^\circ,\alpha_{j,2r-k+1}^\circ)=\left\{\begin{array}{rl}
0 & \quad \mbox{if $j$ is even},\\ -1 & \quad \mbox{if $j$ is odd},
\end{array}\right.\\ \vec{\rho^\circ}(\beta^\circ,\alpha^\circ)&= 
\vec{\rho^\circ}(\alpha_{j,2r-k+1}^\circ,\alpha_{ij}^\circ)=\left\{
\begin{array}{rl} 0 & \quad \mbox{if $j$ is odd}, \\ -1 & \quad
\mbox{if $j$ is even}.  \end{array}\right.
\end{align*}
This yields the formula $\vec{\rho^\circ}(\alpha^\circ,\beta^\circ)-
\vec{\rho^\circ}(\beta^\circ,\alpha^\circ)=(-1)^j$. On the other hand,
in \textbf{\ref{absC3}}, we have seen that 
\[ \vec{\rho}(\alpha,\beta)-\vec{\rho}(\beta,\alpha)=a(-1)^j \qquad
\mbox{where} \qquad  a\in \{1,2,3\}.\]
Hence, we conclude that 
\[\vec{\rho}(\alpha,\beta)-\vec{\rho}(\beta,\alpha)=a\bigl(
\vec{\rho^\circ}(\alpha^\circ,\beta^\circ)-\vec{\rho^\circ}
(\beta^\circ,\alpha^\circ)\bigr).\] 
Using Theorem~\ref{thmF1} and the formula in Example~\ref{expAn}(c), it 
follows again that the desired formula in Proposition~\ref{prop41}(b)
holds. \qed
\end{abs}

\medskip
{\bf Acknowledgements.} I wish to thank Bill Casselman for pointing 
out to me his essay \cite{Cass1}, from which I learned about Rylands' 
unpublished notes \cite{Ry}. This work is a contribution to the SFB-TRR 195
``Symbolic Tools in Mathematics and their Application'' of the German 
Research Foundation (DFG); Project-ID 286237555.


\end{document}